\documentclass[pdflatex,sn-mathphys-num]{sn-jnl}

\usepackage{graphicx}%
\usepackage{multirow}%
\usepackage{amsmath,amssymb,amsfonts}%
\usepackage{amsthm}%
\usepackage{mathrsfs}%
\usepackage[title]{appendix}%
\usepackage{xcolor}%
\usepackage{textcomp}%
\usepackage{manyfoot}%
\usepackage{booktabs}%
\usepackage{algorithm}%
\usepackage{algorithmicx}%
\usepackage{algpseudocode}%
\usepackage{listings}%

\usepackage{amsmath}    
\usepackage{amssymb}    

\usepackage{mathtools}  

\theoremstyle{thmstyleone}%
\newtheorem{theorem}{Theorem}

\theoremstyle{thmstyletwo}%
\newtheorem{remark}{Remark}%

\theoremstyle{thmstylethree}%
\usepackage{amssymb}
\usepackage{amsthm}
\usepackage{tikz}
\usepackage{makeidx}         
\usepackage{graphicx}        
\usepackage{multicol}        
\usepackage[bottom]{footmisc}

\usepackage{newtxtext}       %
\usepackage{newtxmath}       
\usepackage{multirow}
\usepackage{diagbox}

\usepackage{xcolor}

\usepackage{cleveref}
\usepackage{booktabs}
\usepackage{subcaption}

\definecolor{darkgreen}{rgb}{0.0,0.4,0.0}
\definecolor{darkred}{rgb}{0.6,0.0,0.0}
\definecolor{darkblue}{rgb}{0.0,0.0,0.5}
\definecolor{gray}{rgb}{0.5,0.5,0.5}
\definecolor{cyan}{rgb}{0.0,1.0,1.0}
\definecolor{darkcyan}{rgb}{0.0,0.5,0.5}
\definecolor{darkorange}{rgb}{0.8,0.4,0.0}
\definecolor{darkmargenta}{rgb}{0.5,0.0,0.5}
\definecolor{black}{rgb}{0.0,0.0,0.0}

\usepackage{float} 
\usepackage{bm}
\usepackage{comment}

\makeindex             

\def \om {\omega}

\def \calA {\mathcal{A}}

\def \div {\nabla\cdot}

\def \bmv {\bm{v}}

\def \bmw {\bm{w}}

\newcommand{\ip}[2]{\langle #1, #2\rangle}

\newcommand{\Ree}{\operatorname{Re}}

\newtheorem{corollary}[theorem]{Corollary}
\newtheorem{lemma}[theorem]{Lemma}
\theoremstyle{definition}

\newcommand{\vertiii}[1]{{\left\vert\kern-0.17ex\left\vert\kern-0.17ex\left\vert #1 
		\right\vert\kern-0.17ex\right\vert\kern-0.17ex\right\vert}}

\makeatletter
\newsavebox{\@brx}
\newcommand{\llangle}[1][]{\savebox{\@brx}{\(\m@th{#1\langle}\)}%
	\mathopen{\copy\@brx\kern-0.5\wd\@brx\usebox{\@brx}}}
\newcommand{\rrangle}[1][]{\savebox{\@brx}{\(\m@th{#1\rangle}\)}%
	\mathclose{\copy\@brx\kern-0.5\wd\@brx\usebox{\@brx}}}
\makeatother

\usepackage{mdframed}
\usepackage{array}
\usepackage{placeins}
\newcolumntype{C}[1]{>{\centering\arraybackslash}p{#1}}
\usepackage{nomencl}
\makenomenclature

\usepackage{etoolbox}
\renewcommand\nomgroup[1]{%
	\item[\bfseries
	\ifstrequal{#1}{A}{Physical quantities}{%
		\ifstrequal{#1}{N}{Space and time}{%
			\ifstrequal{#1}{B}{Function spaces}{%
				\ifstrequal{#1}{C}{Discretization}{%
					\ifstrequal{#1}{K}{Abbreviations}{}}}}}%
	]}

\usepackage[leftcaption]{sidecap}
\begin{document}
	
	\title{Frequency Domain Biot--Allard Equations for Isotropic and Anisotropic Poroelastic Media: Two-field formulations and iterative splitting}

	\author*[1]{\fnm{Morten} \sur{Jakobsen}}\email{morten.jakobsen@uib.no}
	
	\author[2]{\fnm{Jakob Seierstad} \sur{Stokke}}\email{jakob.stokke@uib.no}
	
	\author[2]{\fnm{Kundan} \sur{Kumar}}\email{kundan.kumar@uib.no}

	\author[2]{\fnm{Florin Adrian} \sur{Radu}}\email{florin.radu@uib.no}

	\affil*[1]{\orgdiv{Center for Modeling of Coupled Subsurface Dynamics}, \orgname{Department of Earth Science, University of Bergen}, \orgaddress{\street{ Allegaten 41}, \city{Bergen}, \country{Norway}}}
	
	\affil[2]{\orgdiv{Center for Modeling of Coupled Subsurface Dynamics}, \orgname{Department of mathematics, University of Bergen}, \orgaddress{\street{ Allegaten 41}, \city{Bergen}, \country{Norway}}}

	\abstract{
	We present a frequency-domain formulation of Biot’s dynamic poroelastic equations with frequency-dependent dissipation (Biot–Allard) for anisotropic, heterogeneous media with memory effects. Two equivalent two-field representations—a displacement–pressure and a velocity–pressure-rate formulation—enable stabilized iterative splitting. While coupling operators generally lack an adjoint or skew-adjoint relationship at finite frequencies, the velocity–pressure-rate representation restores a skew-adjoint structure in the quasi-\color{black}static limit. We prove continuity of the coupling operators and coercivity of the diagonal blocks, essential for convergence of the L-stabilized splitting scheme. The frequency-domain setting eliminates convolutional memory terms, incorporates attenuation and dispersion via complex-valued parameters, and reduces the time-dependent problem to a family of elliptic boundary-value problems suited for parallel computation and multi-frequency inversion. A conforming Galerkin finite element discretization preserves block structure, and numerical experiments confirm robustness and capture frequency-dependent attenuation. To illustrate discretization independence, we include a large-scale wave simulation using a pseudo-spectral method. This work provides a rigorous and efficient framework for modeling wave phenomena in complex porous media.}
	

	\maketitle

	\section{Introduction}

	Porous media play a fundamental role in geosciences because they govern the storage and movement of fluids in the Earth’s crust. In fluid-saturated rocks, mechanical and hydraulic processes are strongly coupled under dynamic conditions, giving rise to phenomena such as wave-induced fluid flow, which significantly affects seismic wave attenuation and dispersion~\cite{pride2004, muller2010seismic}. These interactions are central to understanding subsurface behavior in applications ranging from hydrocarbon recovery and carbon storage to geothermal energy and groundwater management. The complexity of these systems arises from the interplay between pore-scale physics and large-scale geological heterogeneity, which together influence the dynamic response and seismic wave characteristics. Beyond geosciences, similar coupled processes occur in biological tissues, advanced materials, and engineered systems, underlining the interdisciplinary relevance of porous-media modeling. 
	Robust computational methods that are capable of dealing with the multi-scale nature of the governing equations are required for the development of predictive models for such systems. 
	
	Maurice A. Biot laid the foundation for such modeling in the 1940s with his theory of consolidation~\cite{Biot1941}, later extending it to a dynamic theory for the propagation of elastic waves in fluid-saturated porous solids during the 1950s and 1960s~\cite{Biot1956a,Biot1956b,Biot1962}. His theory predicted the existence of two compressional waves—a fast and a slow wave—alongside a shear wave. These predictions were initially met with skepticism but were later confirmed experimentally, most notably by Plona~\cite{Plona1980} in water-saturated sintered glass beads, and subsequently in natural porous media such as sandstones~\cite{Kelder1997, bouzidi2009}. These experimental validations helped establish Biot’s theory as a cornerstone of wave-centric modeling in porous media.
	
	Subsequent extensions, including the work of Allard~\cite{Allard1993}, introduced viscoelastic effects and memory terms to account for frequency-dependent dissipation mechanisms. While Allard's contributions were primarily experimental, the term \emph{Biot--Allard equations} has come to denote Biot's dynamic theory augmented with empirically and theoretically justified frequency-dependent dissipation and dynamic-permeability models, such as the Johnson--Koplik--Dashen (JKD) formulation~\cite{Johnson1987}. The mathematical consistency of these dynamic models and their connection to the quasi-static Biot system have been clarified by Mikeli\'c and Wheeler~\cite{Mikelic2012}, but from a flow-centric perspective and without any numerical results. 
	
	From a computational perspective, time-domain formulations of Biot-type equations are widely used~\cite{carcione2010, Masson2010,Huang2022b, santos1988a}.  Robust and efficient splitting schemes have traditionally been developed to decouple flow and deformation components in the quasi-static, linear Biot system ~\cite{Both2017, Kim2009, kim_undrained2011, mikelic2013, storvik2019}. Extensions to the nonlinear Biot model or the dynamic Biot can be found in \cite{borregales2018, borregales2021, bause2020, kimPHD, kim_2025}. In the computational seismology community, Strang splitting schemes, that are commonly used in quantum mechanics for separating kinetic and potential operators in Schrödinger dynamics, have been adapted to the time-domain Biot equations. In this approach (also referred to as the partitioning method), the evolution operator is decomposed into wave-propagation and diffusive-flow components and integrated using a symmetric sequence of substeps, achieving second-order accuracy and improved stability in seismic wave simulations~\cite{Carcione2022,carcione1995}. While Strang splitting provides a rigorous operator-theoretic foundation based on semigroup theory, it is inherently a time-domain technique. Since our analysis is in the frequency domain, it is more natural to generalize splitting strategies from the flow-centric quasi-static Biot formulation rather than adopt time-domain methods directly. 
	
	Frequency-domain formulations provide an efficient alternative to time-domain approaches for linear partial differential equations, including the dynamic Biot--Allard equations \cite{LiuGreenhalgh2019, Zhan2021, wang2024modelling, jiang2025modelling, yang2021frequency}. Applying a temporal Fourier transform replaces the time derivative $\partial_t$ with an algebraic factor $i\omega$ (depending on the sign convention), reducing the original evolution problem to a family of spatial boundary-value problems parameterized by frequency. For the Biot--Allard system, this transformation converts the coupled hyperbolic--parabolic equations of the time domain into a system of coupled elliptic PDEs in the frequency domain, which is often more convenient for analysis and computation. This representation incorporates attenuation and dispersion through complex-valued material parameters~\cite{jakobsen2003acoustic, jakobsen2009unified}, eliminating convolutional memory terms or fractional derivatives \cite{Carcione2022}. It is particularly advantageous for problems with time-harmonic sources, where the solution is inherently periodic~\cite{Zheng2013, Hu2019}. Frequency-domain methods also simplify the implementation of radiation and absorbing boundary conditions, essential for unbounded domains~\cite{Zhan2021}. They align with the spectral nature of many geophysical data sets and enable multi-frequency inversion strategies that mitigate nonlinearity~\cite{pratt1999seismic, sirgue2004efficient, jakobsen2020transition, yang2021frequency}. Their close connection to Green’s function representations supports integral-equation methods for efficient modeling of seismic scattering in heterogeneous poroelastic media~\cite{Kanaun2018, Huang2024}. Computationally, the ability to solve each frequency independently makes the approach highly parallelizable and well-suited for modern high-performance computing architectures~\cite{LiuGreenhalgh2019, ma2024cost}. These features make frequency-domain formulations an attractive choice for forward modeling and inversion in complex, attenuating media~\cite{LiuGreenhalgh2019, Zhan2021, Kanaun2018, Huang2024, yang2021frequency}.
	
	Among these advantages, the treatment of memory effects is particularly significant. When memory effects are included, time-domain approaches become computationally expensive due to convolutions with long memory kernels or fractional derivatives, as in the JKD model~\cite{Johnson1987,Lu2004}. These operations require storing long histories of particle velocities or solving auxiliary differential equations, significantly increasing cost and complexity~\cite{stokkeEnumath,stokke2024a, stokke_convergence_2025}. In contrast, frequency-domain formulations reduce these convolutions to simple multiplications, facilitating the incorporation of frequency-dependent material properties and streamlining numerical implementation~\cite{Kanaun2018,Huang2024,Carcamo2024}. This advantage has motivated growing interest in frequency-domain modeling for seismic applications~\cite{LiuGreenhalgh2019, Zhan2021, Kanaun2018, Huang2024, yang2021frequency}, where accurate representation of attenuation mechanisms—such as Biot diffusion, squirt flow, and mesoscopic fluid exchange—is essential~\cite{Carcione1998,pride2004}. However, the frequency-domain Biot–Allard system is strongly coupled, indefinite, and complex-valued, complicating stability analysis and solver design.
	
	In this work we will develop and analyse an iterative splitting scheme for solving the dynamic Biot--Allard equations in the frequency domain. To enable a rigorous convergence analysis, we introduce a two-field representation in terms of solid velocity and pressure rate, which restores a skew-adjoint structure in the quasi-static limit and provides a natural Hilbert-space framework for energy estimates. While our main numerical examples for 2D isotropic poroelastic media use the finite element method (FEM) due to its compatibility with complex geometries and operator-theoretic analysis, the proposed L-stabilized splitting scheme operates at the continuous-operator level and is therefore discretization-agnostic. This flexibility allows integration with alternative spatial discretizations, as demonstrated by a large-scale 3D anisotropic example using a pseudo-spectral method (PSM).

	We would like to further emphasize that while L-scheme type splittings are well established for quasi-static, flow-centric poromechanics, our focus is on wave propagation in porous media with memory effects arising from frequency-dependent dissipation mechanisms, such as those introduced through dynamic permeability in Darcy's law. These effects are conveniently handled in the frequency domain, where memory terms are encoded through complex-valued, frequency-dependent coefficients. The resulting Biot--Allard system leads to a coupled elliptic problem with complex-valued fields that must be analyzed in complex Hilbert spaces, giving rise to a coupled block operator system in which standard structural assumptions cannot be taken for granted. In particular, coercivity of the diagonal operators cannot be assumed uniformly in frequency and must be established under additional conditions, and the coupling operators do not, in general, exhibit a simple skew-adjoint or anti-symmetric structure. As a consequence, the operator cannot be treated within the classical energy framework, and the convergence analysis of L-scheme splittings becomes significantly more involved.
	
	To summarize, the main contributions of this paper are as follows: (1) the design and analysis  of a stabilized splitting scheme in the frequency domain, based on a velocity--pressure-rate formulation that restores skew-adjoint structure in the quasi-static limit for the Biot-Allard equations; (2) a general formulation for anisotropic, heterogeneous poroelastic media with memory effects, allowing frequency-dependent and complex-valued permeability and stiffness tensors, with the Johnson--Koplik--Dashen (JKD) model as a representative example; and (3) a wave-centric, discretization-agnostic scheme demonstrated through both finite element simulations and a large-scale three-dimensional anisotropic pseudo-spectral example with a matrix-free, FFT-based implementation. 
	
	
	\color{black}
	
	The paper is organized as follows. In \Cref{sec: 2} three formulations of the dynamic Biot--Allard equations in the frequency domain are introduced and a velocity--pressure-rate representation restoring a skew-adjoint structure in the quasi-static limit is derived. \Cref{sec:operator-weak} develops an operator-theoretic framework with a weak formulation and a Galerkin discretization preserving block structure. In \Cref{sec: 4} we present the stabilized iterative splitting scheme and analyze the convergence under coercivity and continuity conditions. \Cref{sec: numerics} reports numerical experiments in homogeneous and layered media using FEM, followed by a large-scale example based on a PSM to illustrate applications to anisotropic poroelastic media as well as scalability and discretization independence of the stabilized splitting strategy. Section 6 \color{black} summarizes findings and future directions. Appendix A gives the Biot-Allard equations in the time domain. Appendix B discuss some useful identities used in the convergence analysis of the iterative splitting scheme. Appendix C gives coercivity estimates for diagonal blocks; Appendix D explains how fluid mobility terms break skew-adjointness beyond the quasi-static limit and proves continuity properties essential for stability.  Appendix E contains a brief summary of the central ideas behind the PSM and a discussion about how it can be used for matrix-free implementation of the splitting scheme in the case of large-scale 3D models.

	\section{Frequency-Domain Formulations of Dynamic Biot Theory}\label{sec: 2}
	
	\subsection{The Original Four-Field Representation of Biot}

	We start with a formulation equivalent to Biot’s dynamic theory (1962) \cite{Biot1962}, but expressed in a form that makes comparison with alternative representations more transparent. The primary variables are the solid displacement \( \mathbf{u} \), the relative fluid–solid displacement
	\begin{equation}
		\mathbf{w} = \phi(\mathbf{U}-\mathbf{u}),
	\end{equation}
	and the pore pressure \( p \), where \( \phi \) denotes the porosity. Biot’s original formulation was written in terms of the absolute fluid displacement \( \mathbf{U} \) and the solid displacement \( \mathbf{u} \), but he also introduced a related quantity representing relative motion, although not as a primary variable in the same sense as here. Four-field representations involving the \( \mathbf{w} \)-field are widely used in numerical simulations of acoustic and seismic wave propagation in poroelastic media \cite{Carcione2022}.

	We work in the frequency domain, but the underlying governing equations in the time-domain are for completeness shown in Appendix A. 
	Our analysis assumes a fully saturated porous medium composed of a linearly elastic solid matrix and a single-phase Newtonian fluid, under small-strain kinematics. Thermal, chemical, and multiphase effects are neglected. To formulate the governing equations, we adopt the harmonic convention \( e^{i\omega t} \). All field variables are understood as complex amplitudes, and explicit frequency dependence is suppressed except for the dynamic mobility tensor \( \mathbf{B}(\omega) \) and the drained stiffness tensor \( \mathbf{C}(\omega) \), which may exhibit frequency dependence due to viscous and viscoelastic effects. Coefficient fields may be complex-valued to represent inertial and viscous dispersion, and all parameters may vary spatially and exhibit anisotropy.
	
	The governing equations consist of momentum balance for the solid and fluid phases, mass conservation, and constitutive relations. In the frequency-domain, the momentum balance for the solid phase is expressed as
	\begin{equation}
		\nabla \cdot \boldsymbol{\sigma} + \omega^2\big[\rho\,\mathbf{u} + \rho_f\,\mathbf{w}\big] + \mathbf{f} = {\bf 0},
		\label{eq:solid_momentum_mod}
	\end{equation}
	where  \( \boldsymbol{\sigma} \) is the total stress tensor, \( \mathbf{f} \) is an external body-force density and
	\( \rho = (1-\phi)\rho_s + \phi\rho_f \) is the bulk density.
	
	The fluid momentum equation generalizes Darcy’s law to include dynamic effects and takes the form
	\begin{equation}
		i\omega\,\mathbf{w} + \mathbf{B}(\omega)\cdot\big[\nabla p - \omega^2\rho_f\,\mathbf{u}\big] = \mathbf{0},
		\label{eq:fluid_momentum_mod}
	\end{equation}
	where \( \mathbf{B}(\omega)=\boldsymbol{\kappa}(\omega)/\eta \) is the dynamic mobility tensor, expressed in terms of the frequency-dependent permeability tensor \( \boldsymbol{\kappa}(\omega) \) and the fluid viscosity \( \eta \). The dynamic mobility tensor $\mathbf{B}(\omega)$ is a complex, symmetric (Onsager) second-order tensor, $\mathbf{B}(\omega)=\mathbf{B}(\omega)^{\top}$, that is generally non-Hermitian, $\mathbf{B}(\omega)\neq\mathbf{B}(\omega)^{\dagger}$; in the JKD model it is frequency dependent and passive, with $\operatorname{Re}\!\big[\mathbf{v}^{*}\!\cdot\!\mathbf{B}(\omega)\!\cdot\!\mathbf{v}\big]\ge 0$ for all vectors $\mathbf{v}$, while causality under the $e^{i\omega t}$ convention implies $\mathbf{B}(-\omega)=\mathbf{B}(\omega)^{*}$ (Kramers–Kronig consistency).

	Mass conservation couples the pore pressure, solid displacement, and relative displacement according to
	\begin{equation}
		i\omega \left[ \frac{p}{M}
		+ \boldsymbol{\alpha}:\boldsymbol{\varepsilon}(\mathbf{u})
		+ \nabla\cdot\mathbf{w}\right] - q = 0,
		\label{eq:mass_balance_mod}
	\end{equation}
	where $M$ is the Biot modulus, a scalar parameter that characterizes the compressibility of the fluid–solid system at constant strain~\cite{Carcione2022}, and $\boldsymbol{\alpha}$ is the effective stress tensor. Here, $q$ denotes the rate of fluid-content change per unit bulk volume (positive for injection), and we will also use the notation $\beta = 1/M$.

	For a general anisotropic porous medium composed of isotropic minerals, the Biot modulus $M$ can be expressed as
	\begin{equation}
		M =
		\frac{K_s}{
			(1 - K^{*}/K_s)
			-
			\phi \left(1 - K_s/K_f\right)
		},
	\end{equation}
	where $K_f$ and $K_s$ denote the bulk moduli of the fluid and solid phases, respectively. The quantity $K^{*}$ is the rotational invariant corresponding to the Voigt-average bulk modulus of the drained stiffness tensor,
	$K^* = \frac{1}{9}\, C^{(m)}_{iijj}$.
	This expression is equivalent to Eq.~(7.172) in~\cite{Carcione2022}.
	
	The Biot effective stress tensor is defined by (see Eq.~7.165 in~\cite{Carcione2022})
	\begin{equation}
		\alpha_{ij}
		=
		\delta_{ij}
		-
		\frac{1}{K_s} C^{(m)}_{ijkl},
		\label{eq_alpha}
	\end{equation}
	where $C^{(m)}_{ijkl}$ denotes the drained elastic stiffness tensor of the porous skeleton. It follows directly that $\boldsymbol{\alpha}$ is symmetric, i.e. $\alpha_{ij} = \alpha_{ji}$, since $C^{(m)}_{ijkl}$ satisfies the minor symmetry $C^{(m)}_{ijkl} = C^{(m)}_{jikl}$.
	
	\color{black}
	
	The constitutive relation for the total stress is given by
	\begin{equation*}
		\boldsymbol{\sigma} = \mathbf{C}(\omega):\boldsymbol{\varepsilon} - \boldsymbol{\alpha}\,p,
		\qquad
		\boldsymbol{\varepsilon}(\mathbf{u}) = \tfrac{1}{2}\big(\nabla\mathbf{u} + \nabla\mathbf{u}^\top\big),
	\end{equation*}
	where $\mathbf{C}(\omega)$ denotes the drained stiffness tensor. We adopt the notation of Auld~\cite{Auld1973}, where the symbol $:$ denotes the double contraction of tensors, i.e.
	\[
	(\mathbf{C} : \boldsymbol{\varepsilon})_{ij} = C_{ijkl}\,\varepsilon_{kl}.
	\]
	The negative sign in the coupling term reflects Terzaghi’s principle: an increase in pore pressure reduces the effective stress carried by the solid matrix. The stiffness tensor $\mathbf{C}(\omega)$ satisfies the usual symmetry and positivity (or positive-real) conditions. Effects of squirt flow are not included in the dynamic Biot model but may be represented through a frequency-dependent viscoelastic stiffness tensor~\cite{pride2004, muller2010seismic, jakobsen2009unified}.

	According to Neumann's symmetry principle~\cite{nye1985physical}, effective tensors associated with a given microstructure must share the same symmetry class. In a consistent poroelastic model, this implies that tensors such as the elastic stiffness, mobility, and coupling coefficients inherit the same symmetry properties, although in practice different levels of approximation may be used depending on the modeling assumptions \cite{jakobsen2009velocity}.

	It is worth noting that the frequency-domain formulation naturally incorporates memory effects through the frequency dependence of the constitutive tensors, such as \( \mathbf{B}(\omega) \) and \( \mathbf{C}(\omega) \). Modeling these effects in the time domain would require convolution integrals with long memory kernels, making the formulation significantly more complex and computationally demanding.

	\subsection{A Displacement--Pressure Two-Field Representation}
	
	An alternative to the three-field formulation is a two-field representation in terms of the solid displacement \( \mathbf{u} \) and the pore pressure \( p \). This system is obtained by eliminating the relative displacement \( \mathbf{w} \) from the governing equations using the dynamic Darcy law (\ref{eq:fluid_momentum_mod}):
	\begin{equation}
		\mathbf{w} = \frac{i}{\omega}\,\mathbf{B}(\omega)\cdot\big[\nabla p - \omega^2\rho_f\,\mathbf{u}\big].
		\label{eq:w_expression}
	\end{equation}
	Substituting \eqref{eq:w_expression} into the solid momentum 
	equation (\ref{eq:solid_momentum_mod}) gives 
	\begin{equation}
		\omega^2\,\boldsymbol{\rho}_{\mathrm{eff}}\cdot\mathbf{u} + \nabla\cdot\big[\mathbf{C}(\omega):\boldsymbol{\varepsilon}\big]
		+ \mathbf{s}_u + \mathbf{f} = {\bf 0},
		\label{displacement_eq}
	\end{equation}
	where 
	\begin{equation}
		\boldsymbol{\rho}_{\mathrm{eff}} = \rho\,\mathbf{I} - i\omega\,\rho_f^2\,\mathbf{B}(\omega),
		\label{def_rho_eff}
	\end{equation}
	is the effective mass density tensor, and
	\begin{equation*}
		\mathbf{s}_u = -\,\nabla\cdot(\boldsymbol{\alpha}\,p) + i\omega\,\rho_f\,\mathbf{B}(\omega)\cdot\nabla p.
	\end{equation*}
	is a pressure related coupling-source term in the displacement equation. 
	In equation (\ref{def_rho_eff}), $\mathbf{I}$ denotes the second-rank unit tensor.

	The pressure equation is obtained by computing \( \nabla\cdot\mathbf{w} \) from \eqref{eq:w_expression} and substituting it into the mass balance equation (\ref{eq:mass_balance_mod}):
	\begin{equation}
		i\omega \frac{p}{M} - \nabla\cdot\big[\mathbf{B}(\omega)\cdot\nabla p\big] + s_p - q = 0,
		\label{pressure_eq}
	\end{equation}
	where
	\[
	s_p = \,i\omega\,\boldsymbol{\alpha}:\epsilon(\mathbf{u}) +\omega^2\rho_f\,\nabla\cdot\big(\mathbf{B}(\omega)\cdot\mathbf{u}\big),
	\]
	is a displacement related coupling-source term in the pressure equation.
	
	This formulation is mathematically equivalent to Biot’s original four-field formulation and is related to what is often called the Biot--Allard formulation in the acoustics literature~\cite{Allard1993, Mikelic2012}. However, the classical Biot--Allard equations are formulated in the time domain and incorporate memory effects only in the fluid-flow term. To the best of our knowledge, the present frequency-domain version has not been widely documented, apart from the work of Huang et al.\ in \emph{Modeling of seismic wave scattering in poroelastic media} \cite{Huang2024}, where a similar \( (\mathbf{u},p) \)-formulation was introduced in the context of Green’s functions and integral representations. While this approach is less common in geophysical modeling, it provides a useful reference point and is employed in some frequency-domain finite element methods. However, the coupling operators lack an adjoint relationship because they involve different powers of \( \omega \), which complicates energy estimates and variational formulations. To address this, we now derive an equivalent velocity--pressure-rate representation.

	\subsection{A Velocity-Pressure Rate Two-Field Representation}
	
	To obtain a two-field system with more symmetric coupling operators, we introduce the solid velocity and the pressure rate:
	\[
	\mathbf{v} = i\omega\,\mathbf{u}, 
	\qquad 
	\dot{p} = i\omega\,p,
	\]
	Starting from the displacement--pressure formulation, we multiply the displacement equation (\ref{displacement_eq}) by \( i\omega \) and use the above relations so that the form of the equation is preserved but with ${\bf u}$ is replaced by ${\bf v}$, $p$ is replaced by $\dot{p}$ and ${\bf f}$ is replaced by $\dot{\mathbf{f}} = i\omega \mathbf{f}$:
	\begin{equation}
		\omega^2\,\boldsymbol{\rho}_{\mathrm{eff}}\cdot\mathbf{v} + \nabla\cdot\big[\mathbf{C}(\omega):\boldsymbol{\varepsilon}(\mathbf{v})\big] +  \mathbf{s}_v + \dot{\mathbf{f}} = {\bf 0},
		\label{eq:vpr_velocity}
	\end{equation}
	where
	\begin{equation*}
		\mathbf{s}_v = -\,\nabla\cdot(\boldsymbol{\alpha}\,\dot{p}) + i\omega\,\rho_f\,\mathbf{B}(\omega)\cdot\nabla\dot{p}.
	\end{equation*}

	Applying the same change of variables to the pressure equation (\ref{pressure_eq}), but without multiplying by $i\omega $, we get the following pressure rate equation: 
	\begin{equation}
		\frac{\dot{p}}{M} - \frac{1}{i\omega}\,\nabla\cdot\big[\mathbf{B}(\omega)\cdot\nabla\dot{p}\big] + s_{\dot{p}} - q = 0,
		\label{eq:vpr_pressure}
	\end{equation}
	where
	\begin{equation*}
		s_{\dot{p}} = \boldsymbol{\alpha}:\epsilon(\mathbf{v}) - i\omega\,\rho_f\,\nabla\cdot\big(\mathbf{B}(\omega)\cdot\mathbf{v}\big).
	\end{equation*}
	
	After this change of variables, the coupling terms 
	$\mathbf{s}_v$ and $s_{\dot{p}}$ constitute the off--diagonal 
	blocks $A_{12}$ and $A_{21}$ of the operator matrix introduced 
	in the next section. 
	A natural question is: why does the quasi-static limit of the Biot--Allard $(\mathbf{u},p)$ 
	equations have a skew-adjoint coupling structure in the time domain, but apparently not in the frequency domain? 
	By the \emph{time--domain Biot--Allard equations} we mean the classical 
	two--field formulation in $\mathbf{u}(t)$ and $p(t)$, with memory 
	effects restricted to the Darcy flow term. 
	The natural Hilbert space inner product is defined by the sum of the 
	kinetic and storage energies, which involve the solid velocity 
	$\partial_t \mathbf{u}$ and the pressure rate $\partial_t p$. 
	With respect to this energy inner product, the coupling operators are 
	adjoint. 
	After transforming to the frequency domain but keeping 
	$(\mathbf{u},p)$ as the primary variables, these time derivatives 
	become $i\omega\,\mathbf{u}$ and $i\omega\,p$, so the corresponding 
	energy inner product acquires explicit frequency factors that obscure 
	the skew-adjoint coupling structure; which is known to exist in the the quasi-static limit.
	By instead taking $(\mathbf{v},\dot{p})$ as the unknowns, these 
	$\omega$--dependent scalings are absorbed into the variables, and the 
	frequency--domain formulation preserves the $L^{2}$--skew-adjoint coupling 
	inherited from the time--domain energy inner product in the quasi-static limit. 
	This restored symmetry will be fundamental for the Hilbert space 
	framework and stability analysis developed below.

	\section{Operator Framework and Finite Element Discretization}
	\label{sec:operator-weak}
	
	Biot’s equations in the frequency domain form a coupled PDE system that is challenging to analyze and discretize directly. We adopt the \emph{velocity–pressure rate formulation}, which simplifies inertial terms and reveals a block structure suitable for energy analysis, well-posedness, and stable discretization. 
	Casting the system in an \emph{operator framework} provides two key benefits: 
	(i) a rigorous Hilbert space setting for stability analysis, and 
	(ii) a natural basis for variational formulations and iterative methods such as the $L$-stabilized splitting scheme.
	
	\subsection{Notation for Operator-Theoretic Analysis}
	
	Throughout this paper, we let $\Omega \subset \mathbb{R}^d$ be a bounded Lipschitz domain.  We will use standard notation from functional analysis, so we let $L^{2}(\Omega;\mathbb{C})$ be the space of all complex square-integrable functions.  Further, we let $H^{1}(\Omega;\mathbb{C})$ be the complex Sobolev space with first-order weak derivatives in $L^{2}(\Omega;\mathbb{C})$. By $H^{1}_{0}(\Omega;\mathbb{C})$ we denote the subspace of $H^{1}(\Omega;\mathbb{C})$ with vanishing trace at the boundary. Let $\mathscr{D}'(\Omega;\mathbb{C})$ be the space of distributions. To shorten notation, we omit the domain $\Omega$ and since all spaces are complex we skip $\mathbb{C}$ from the notation, i.e. $L^{2}(\Omega;\mathbb{C})=L^{2}$. For vector-valued functions, the spaces are written in bold, i.e. $(L^{2})^{d}=\bm{L^{2}}$.
	The $L^{2}$ and $\bm{L}^{2}$ inner products are complex conjugate in the second argument, i.e.
	\[
	\langle \mathbf{v}, \mathbf{w} \rangle_{\bm{L^{2}}} = \int_\Omega \mathbf{v}\cdot\overline{\mathbf{w}}\,\mathrm{d}\mathbf{x}, \qquad
	\langle \dot{p}, r \rangle_{L^{2}} = \int_\Omega \dot{p}\,\overline{r}\,\mathrm{d}\mathbf{x}.
	\]
	The associated norms of $L^{2}$ or $\bm{L}^{2}$ will be represented by $\|\cdot\|:=\|\cdot\|_{L^{2}}$, norms associated with other spaces will be denoted in the subscript. By $V$ we denote the product space $V = \bm{L}^{2} \times L^{2}$.
	
	\subsection{Block Operator Formulation}
	
	The governing equations can be written as a block operator system:
	\[
	\mathcal{A}(\omega)
	\begin{pmatrix}
		\mathbf{v} \\
		\dot{p}
	\end{pmatrix}
	=
	\begin{pmatrix}
		\mathcal{A}_{11} & \mathcal{A}_{12} \\
		\mathcal{A}_{21} & \mathcal{A}_{22}
	\end{pmatrix}
	\begin{pmatrix}
		\mathbf{v} \\
		\dot{p}
	\end{pmatrix}
	=
	\begin{pmatrix}
		\dot{\mathbf{f}} \\
		q
	\end{pmatrix}.
	\]
	The diagonal blocks represent elastic and storage/diffusion contributions, while the off-diagonal blocks capture inertial and coupling effects. Their explicit forms follow the velocity–pressure rate formulation \eqref{eq:vpr_velocity}-\eqref{eq:vpr_pressure} of dynamic Biot theory:
	\begin{subequations}\label{eq:block-operators-vpr}
		\begin{align}
			\mathcal{A}_{11}\mathbf{v} &=
			-\,\omega^2\, \boldsymbol{\rho}_{\mathrm{eff}}\cdot \mathbf{v}
			- \nabla \cdot\big( \mathbf{C}(\omega) : \boldsymbol{\varepsilon}(\mathbf{v}) \big),\\
			\mathcal{A}_{12}\dot{p} &=
			\nabla \cdot\big( \boldsymbol{\alpha}\, \dot{p} \big)
			- i\omega\, \rho_f\, \mathbf{B}(\omega)\cdot \nabla \dot{p},\\
			\mathcal{A}_{21}\mathbf{v} &=
			\boldsymbol{\alpha}:\epsilon(\mathbf{v})
			- i\omega\, \rho_f\, \nabla \cdot\big( \mathbf{B}(\omega)\cdot \mathbf{v} \big),\\
			\mathcal{A}_{22}\dot{p} &=
			\frac{\dot{p}}{M}
			- \frac{1}{i\omega}\, \nabla \cdot\big( \mathbf{B}(\omega)\cdot \nabla \dot{p} \big).
		\end{align}
	\end{subequations}
	The diagonal blocks $\mathcal{A}_{11}$ and $\mathcal{A}_{22}$ represent, respectively, the elastic–inertial response of the solid skeleton and the storage–diffusion behavior of the fluid phase. The off-diagonal blocks $\mathcal{A}_{12}$ and $\mathcal{A}_{21}$ capture coupling between these phases through momentum exchange and pressure-induced deformation, including energy dissipation via mobility terms.
	
	The block operator acts on the $L^2$-based Hilbert product $V$ with domain
	\[
	D(\mathcal{A}) \subset \bm{H}_0^{1} \times H^1,
	\]
	ensuring the weak \color{black} derivatives are well-defined while adjoint structure and energy estimates use $L^2$ pairings.
	
	\medskip
	\begin{itemize}
		\item \textbf{Strong ellipticity:} Under standard assumptions on $\mathbf{C}(\omega)$ and $\mathbf{B}(\omega)$, these operators are strongly elliptic in the sense of second-order PDEs (see \cite{Bonnet2016}), which underpins well-posedness.
		\item \textbf{Coercivity:} Coercivity is established using $L^2$ inner products and Poincar\'e’s inequality to control gradient terms. $\mathcal{A}_{22}$ is uniformly coercive due to the storage term, whereas $\mathcal{A}_{11}$ is coercive only for angular frequencies below a critical threshold depending on material parameters and dynamic mobility.
		\item \textbf{Boundedness:} The coupling operators are continuous as maps
		\[
		\mathcal{A}_{12}: H^1\to \bm{L}^{2},\qquad
		\mathcal{A}_{21}: \bm{H}^{1}\to L^2,
		\]
		with norms growing at most linearly in $|\omega|$. Boundedness follows from multiplication estimates using $L^\infty$ bounds of coefficients and the weak product rule.
		\item \textbf{Skew-adjoint structure (quasi-static limit):} Under the chosen $L^2$ inner product, $\mathcal{A}_{12}$ and $\mathcal{A}_{21}$ satisfy a skew-adjoint relationship in the quasi-static limit. Beyond this limit, energy dissipation caused by fluid-solid friction breaks this symmetry.
	\end{itemize}
	Formal statements and proofs of these properties, essential for the splitting solution strategy, are provided in \Cref{sec: appendix diagonal} and \Cref{sec: appendix A}.

	\subsection{Weak Formulation}
	In this section, we present the weak formulation of the block operator system in the previous section \eqref{eq:block-operators-vpr}. We also consider the finite element method for discretizing the weak formulation in space. 
	
	We assume that the coefficients $\boldsymbol{\rho}_{\text{eff}}$, $\mathbf{C}$, $\boldsymbol{\alpha}$, $\rho_f$, $\mathbf{B}$, and $\boldsymbol{\beta}$ satisfy the regularity and boundedness conditions stated in Section~2.3.  Then, the weak formulation of the block operator system~\eqref{eq:block-operators-vpr}
	is: find $(\mathbf{v}, \dot{p}) \in \bm{H}_{0}^{1} \times H_{0}^{1}$ such that,
	\begin{equation}\label{eq: weak form}
		\begin{aligned}
			-\,\langle \omega^2 \boldsymbol{\rho}_{\text{eff}} \mathbf{v}, \mathbf{w} \rangle
			+ \langle \mathbf{C} : \boldsymbol{\varepsilon}(\mathbf{v}), \boldsymbol{\varepsilon}(\mathbf{w}) \rangle
			+ \langle \div(\boldsymbol{\alpha} \dot{p}), \mathbf{w} \rangle
			- i\omega \langle \rho_f \mathbf{B} \nabla \dot{p}, \mathbf{w} \rangle
			= \langle \dot{\mathbf{f}}, \mathbf{w} \rangle, \\
			\langle \boldsymbol{\alpha}: \boldsymbol{\varepsilon}(\mathbf{v}),  r \rangle
			- i\omega \langle \rho_f \mathbf{B} \mathbf{v}, \nabla r \rangle
			- \langle \operatorname{tr}[\boldsymbol{\beta}] \dot{p}, r \rangle
			- \tfrac{1}{i\omega} \langle \mathbf{B} \nabla \dot{p}, \nabla r \rangle
			= \langle \dot{q}, r \rangle,
		\end{aligned}
	\end{equation}
	for all $(\mathbf{w}, r) \in \left(\bm{H}_{0}^{1} \times H_{0}^{1}\right)$.
	

	
	For the spatial discretization, we use a conforming Galerkin method that preserves the operator block structure. Let $\mathcal{T}_h$ be a shape-regular mesh of $\Omega$ with mesh size $h:= \max_{K\in\mathcal{T}_h}({\rm diam}(K))$. 
	We consider conforming subspaces
	\[
	\bm{V}_h \subset \bm{H}_{0}^{1}, 
	\qquad 
	Q_h \subset H_{0}^{1},
	\]
	with bases $\{\boldsymbol{\varphi}_i^V\}_{i=1}^{n_{{{v}}}}$ and $\{\varphi_j^Q\}_{j=1}^{n_Q}$. 
	We use test functions $\boldsymbol{\Phi}_h^V \in \bm{V}_h$ and $\Phi_h^{{Q}} \in Q_h$.
	The finite element discretization of \eqref{eq: weak form} is then: 
	Find $\mathbf{v}_h \in \bm{V}_h$ such that for all $\boldsymbol{\Phi}_h^V \in \bm{V}_h$,
	\begin{align*}
		-\langle \omega^2\, \boldsymbol{\rho}_{\mathrm{eff}}\, \mathbf{v}_h, \boldsymbol{\Phi}_h^V \rangle
		+ \big\langle \mathbf{C} : \boldsymbol{\varepsilon}(\mathbf{v}_h), \boldsymbol{\varepsilon}(\boldsymbol{\Phi}_h^V) \big\rangle
		+ \langle \div(\boldsymbol{\alpha}\, \dot{p}_h), \boldsymbol{\Phi}_h^V \rangle
		- i\omega\, \langle \rho_f\, \mathbf{B}\, \nabla \dot{p}_h, \boldsymbol{\Phi}_h^V \rangle
		\\= \langle \dot{\mathbf{f}}, \boldsymbol{\Phi}_h^V \rangle,
	\end{align*}
	and find $\dot{p}_h \in Q_h$ such that for all $\Phi_h^Q \in Q_h$,
	\begin{align*}
		\langle \boldsymbol{\alpha}: \boldsymbol{\varepsilon}(\mathbf{v}_h),  \Phi_h^Q \rangle
		- i\omega\, \langle \rho_f\, \mathbf{B}\, \mathbf{v}_h, \nabla \Phi_h^Q \rangle
		- \langle \operatorname{tr}[\boldsymbol{\beta}]\, \dot{p}_h, \Phi_h^Q \rangle
		- \frac{1}{i\omega}\, \langle \mathbf{B}\, \nabla \dot{p}_h, \nabla \Phi_h^Q \rangle
		= \langle \dot{q}, \Phi_h^Q \rangle.
	\end{align*}

	The discrete problem above can be written as a block operator system by expanding the discrete trial fields. 
	Then, using standard elementwise assembly (tensor contractions and complex conjugation in inner products are understood), we define
	\begin{align*}
		[M_{\rho_{\mathrm{eff}}}]_{j i} &:= \langle \boldsymbol{\rho}_{\mathrm{eff}}\, \boldsymbol{\varphi}_i^V, \boldsymbol{\varphi}_j^V \rangle, 
		&
		[K_C]_{j i} &:= \langle \mathbf{C} : \boldsymbol{\varepsilon}(\boldsymbol{\varphi}_i^V), \boldsymbol{\varepsilon}(\boldsymbol{\varphi}_j^V) \rangle, \\[2pt]
		[M_\beta]_{j i} &:= \langle \operatorname{tr}[\boldsymbol{\beta}]\, \varphi_i^Q, \varphi_j^Q \rangle, 
		&
		[K_B]_{j i} &:= \langle \mathbf{B}\, \nabla \varphi_i^Q, \nabla \varphi_j^Q \rangle, \\[2pt]
		[E_\alpha]_{j i} &:= \langle \div(\boldsymbol{\alpha}\, \varphi_i^Q), \boldsymbol{\varphi}_j^V \rangle, 
		&
		[Q_\alpha]_{j i} &:= \langle \boldsymbol{\alpha}: \boldsymbol{\varepsilon}(\boldsymbol{\varphi}_i^V) ,  \varphi_j^Q \rangle, \\[2pt]
		[Q_B]_{j i} &:= \langle \rho_f\, \mathbf{B}\, \boldsymbol{\varphi}_i^V, \nabla \varphi_j^Q \rangle,
		&
		[G_B]_{j i} &:= \langle \rho_f\, \mathbf{B}\, \nabla \varphi_i^Q, \boldsymbol{\varphi}_j^V \rangle.
	\end{align*}
	The discrete right-hand sides use the same notation as in the continuous weak formulation:
	\[
	[\mathbf{F}]_j := \langle \dot{\mathbf{f}}, \boldsymbol{\varphi}_j^V \rangle,
	\qquad
	[\mathbf{S}]_j := \langle \dot{q}, \varphi_j^Q \rangle.
	\]
	Collecting coefficients into vectors $\mathbf{V} \in \mathbb{C}^{n_v}$ and $P \in \mathbb{C}^{n_Q}$, the global block system reads
	\[
	\mathbf{A}(\omega)
	\begin{pmatrix}
		\mathbf{V} \\
		P
	\end{pmatrix}
	=
	\begin{pmatrix}
		\mathbf{F} \\
		\mathbf{S}
	\end{pmatrix},
	\qquad
	\mathbf{A}(\omega) =
	\begin{pmatrix}
		\mathbf{A}_{11}(\omega) & \mathbf{A}_{12}(\omega) \\
		\mathbf{A}_{21}(\omega) & \mathbf{A}_{22}(\omega)
	\end{pmatrix},
	\]
	with block components
	\[
	\begin{aligned}
		\mathbf{A}_{11} &= -\,\omega^2\, M_{\rho_{\mathrm{eff}}} + K_C, \\
		\mathbf{A}_{12} &= \,E_\alpha \;-\; i\omega\, G_B,\\
		\mathbf{A}_{21} &= \phantom{-}\,Q_\alpha \;-\; i\omega\, Q_B, \\
		\mathbf{A}_{22} &= -\,M_\beta \;-\; \tfrac{1}{i\omega}\, K_B.
	\end{aligned}
	\]
	
	\noindent
	\begin{remark}
		(i) The ordering of terms in the discrete weak form mirrors the continuous weak formulation to make the correspondence transparent.  
		(ii) The pair $(E_\alpha,Q_\alpha)$ realizes the anisotropic $\boldsymbol{\alpha}$-coupling in the mechanical and fluid equations, respectively. Under homogeneous boundary conditions and Hermitian coefficients, these satisfy the expected skew-adjoint relationship at the discrete level (in the quasi-static limit), $E_\alpha = Q_\alpha^*$, and likewise $G_B = Q_B^*$.  
		(iii) If desired, one may equivalently assemble the fluid mobility term in divergence form using $+\,i\omega\,\langle \nabla\cdot(\rho_f \mathbf{B}\,\boldsymbol{\varphi}_i^V), \varphi_j^Q \rangle$; the block structure above is unchanged.
	\end{remark}

	\section{An Iterative Splitting Scheme for the Block Operator System}\label{sec: 4}
	
	Solving the block matrix system as a single monolithic problem can be computationally demanding and memory-intensive, particularly for large-scale heterogeneous models. To address this, we propose a splitting strategy that is \emph{discretization-agnostic}: it operates at the algebraic level and does not depend on the choice of spatial discretization. This flexibility allows the scheme to be applied to finite element, finite difference, finite volume or pseudo-spectral methods alike. Although our analysis focuses on the finite element method (FEM), the numerical examples include a case based on a pseudo-spectral method (PSM) for spatial discretization, illustrating the generality of the approach. The iterative scheme alternates between velocity and pressure updates and incorporates a stabilization term to ensure robust convergence.

	Given the block matrix system, we use a splitting scheme described in \Cref{alg: iterative}. 
	
	\begin{algorithm}[H]
		\caption{Iterative splitting scheme }\label{alg: iterative}
		\begin{algorithmic}
			\Require $\boldsymbol{V}^{0},\, P^{0}$ as initial guess. 
			\For{k=1,2,..}
			\State Step 1: Given $P^{k}$ find $\boldsymbol{V}^{k+1}$ such that
			\begin{align*}
				\mathcal{A}_{11} \mathbf{V}^{(k+1)} + 
				\mathcal{A}_{12} P^{(k)}  + 
				L\left( \mathbf{V}^{(k+1)} - \mathbf{V}^{(k)}\right)
				= \mathbf{F}  
			\end{align*}	
			\State Step 2: Given $\boldsymbol{V}^{k+1}$ find $P^{k+1}$ such that
			\begin{align*}
				\mathcal{A}_{22} P^{(k+1)} + 
				\mathcal{A}_{21} \mathbf{V}^{(k+1)} = \mathbf{S}
			\end{align*}
			\State Repeat Step 1 and Step 2 until the changes in $\boldsymbol{V}$ and $P$ are below a specified tolerance level.
			\EndFor
		\end{algorithmic}
	\end{algorithm}

	\begin{theorem}[Convergence of the L-stabilized splitting scheme in the general case with non-skew-adjoint coupling terms]\label{thm: main theorem}
		Assume the setting above, in particular the coercivity conditions for $\mathcal{A}_{11}$ and $\mathcal{A}_{22}$. By defining
		\begin{equation}\label{eq: Lmin}
			L_{\min} \;:=\; 2\frac{\|\mathcal{A}_{12}\|^2 \|\mathcal{A}_{21}\|^2}{\alpha_{11}\,\alpha_{22}^2}.
		\end{equation}
		If $L > L_{\min}$, then the stabilized splitting scheme converges. Moreover, the error sequence satisfies
		\[
		\|\mathbf{e}_1^{k+1}\| \le \rho \,\|\mathbf{e}_1^{k}\|,
		\qquad
		\rho := \sqrt{\frac{L}{L+\alpha_{11}}} < 1,
		\]
		so the method is linearly convergent with contraction factor $\rho$. 
	\end{theorem}

	\begin{proof}[Proof] 
		The proof is organized into five conceptual steps for readability.

		\noindent\textbf{Step 1 (Error equations and energy identity).} \\
		We denote the error between the iteration and true solution by $\bm{e}_{1}^{k}=\bmv^{k}-\bmv$ and $e^{k}_2=\dot{p}^{k}-\dot{p}$. Then by subtracting the true solution from the iterative procedure at $k+1$, we obtain
		\begin{subequations}
			\begin{align}
				\mathcal{A}_{11}\bm{e}_{1}^{k+1}+\mathcal{A}_{12}e_2^k+L(\bm{e}_{1}^{k+1}-\bm{e}_{1}^{k})&=0,\label{eq: error equation one}\\
				\mathcal{A}_{22}e_2^{k+1}+\mathcal{A}_{21}\bm{e}_{1}^{k+1}&=0.\label{eq: error equation two}
			\end{align}
		\end{subequations}
		
		In the first error equation \eqref{eq: error equation one} we test with $\mathbf{e}_1^{k+1}$ and in the second error equation \eqref{eq: error equation two} with $e_2^{k+1}$. Then, by adding them, we get
		\[
		\ip{\mathcal{A}_{11}\mathbf{e}_1^{k+1}}{\mathbf{e}_1^{k+1}}
		+ \ip{\mathcal{A}_{12} e_2^{k}}{\mathbf{e}_1^{k+1}}
		+ \ip{L(\mathbf{e}_1^{k+1}-\mathbf{e}_1^{k})}{\mathbf{e}_1^{k+1}}
		+ \ip{\mathcal{A}_{22}e_2^{k+1}}{e_2^{k+1}}
		+ \ip{\mathcal{A}_{21}\mathbf{e}_1^{k+1}}{e_2^{k+1}}=0.
		\]
		Next, we take the real part of the above equation and use the coercivity assumption on $\mathcal{A}_{11}$ and $\mathcal{A}_{22}$. We then have that 
		\begin{equation}\label{eq:energyline}
			\alpha_{11}\|\mathbf{e}_1^{k+1}\|^2 + \alpha_{22}\|e_2^{k+1}\|^2
			+ \Ree\,\ip{L(\mathbf{e}_1^{k+1}-\mathbf{e}_1^{k})}{\mathbf{e}_1^{k+1}}
			+ \Ree\,M_k \;\le\; 0,
		\end{equation}
		where
		\begin{equation}\label{eq: Mk}
			M_k := \ip{\mathcal{A}_{12} e_2^{k}}{\mathbf{e}_1^{k+1}} + \ip{\mathcal{A}_{21}\mathbf{e}_1^{k+1}}{e_2^{k+1}}.
		\end{equation}
		The first two terms in \eqref{eq:energyline} can be interpreted as the energy, and the third is the stabilization increment. The last term $M_k$ is the coupling.
	
		\noindent\textbf{Step 2 (Exact decomposition of the non-skew adjoint coupling).} \\
		\color{black}
		We now analyze the coupling term \eqref{eq: Mk}
		which arises in the energy identity. This term couples the two blocks of the system and is in general not skew-adjoint. To estimate it effectively, we decompose it into two parts: a difference term and an instantaneous coupling term. By inserting and subtracting \( \mathcal{A}_{12} e_2^{k+1} \),
		and then taking real parts of \eqref{eq: Mk}, we obtain 
		\begin{equation}\label{eq:Msplit}
			\Ree\,M_k
			= \Ree\,\ip{\mathcal{A}_{12}(e_2^{k}-e_2^{k+1})}{\bm{e}_{1}^{k+1}}
			+ \Ree\,\big(\ip{\mathcal{A}_{12} e_2^{k+1}}{\bm{e}_{1}^{k+1}} + \ip{\mathcal{A}_{21} \bm{e}_{1}^{k+1}}{e_2^{k+1}}\big).
		\end{equation}
		The first term involves the difference \( e_2^k - e_2^{k+1} \), which will be controlled via stabilization and bounded using Young’s inequality. The second term can be controlled by considering the actual terms involved in $\mathcal{A}_{12}$ and $\mathcal{A}_{21}$.

		\noindent\textbf{Step 3 (Analysis of the instantaneous coupling term).} \\
		We consider the instantaneous coupling term in the energy identity for the dynamic Biot-type system with frequency-dependent fluid mobility tensor $\mathbf{B}(\omega)$. The term reads
		\[
		\langle A_{12} e^{k+1}_2, \mathbf{e}^{k+1}_1 \rangle 
		+ \langle A_{21} \mathbf{e}^{k+1}_1, e^{k+1}_2 \rangle.
		\]
		Explicit expressions for the coupling operators are given in the paper. Our goal is to show that this term does not contribute to the \emph{real part} of the energy.
		
		\medskip
		
		\noindent
		\textbf{Effective stress contribution.}
		By integration by parts and using the symmetry of the tensor $\boldsymbol{\alpha}$, we obtain
		\[
		\langle \nabla \cdot (\boldsymbol{\alpha} e^{k+1}_2), \mathbf{e}^{k+1}_1 \rangle
		+\langle \boldsymbol{\alpha} : \boldsymbol{\varepsilon}(\mathbf{e}^{k+1}_1), e^{k+1}_2 \rangle = 0.
		\]
		This identity reflects the skew-adjoint structure of the quasi-static $\boldsymbol{\alpha}$-coupling. Hence, these terms do not contribute to the real part of the energy.
		
		\medskip
		
		\noindent
		\textbf{Mobility contribution.}
		We now consider the terms involving the mobility tensor. Define
		\[
		\Theta  := -i\omega \rho_f \Big( \langle \mathbf{B} \cdot \nabla e_2^{k+1}, \mathbf{e}_1^{k+1} \rangle + \langle \nabla \cdot (\mathbf{B} \cdot \mathbf{e}_1^{k+1}), e_2^{k+1} \rangle \Big).
		\]
		Using integration by parts and the sesquilinear inner product, this becomes
		\begin{equation}
			\Theta  = -i\omega \rho_f \int_\Omega \Big[ (\mathbf{B} \cdot \nabla e_2^{k+1}) \cdot \overline{\mathbf{e}_1^{k+1}} - (\mathbf{B} \cdot \mathbf{e}_1^{k+1}) \cdot \overline{\nabla e_2^{k+1} }\Big] \, dx.
			\label{Expression_Theta}
		\end{equation}

		\medskip
		
		\noindent
		We now analyze $\mathrm{Re}(\Theta)$ by considering different cases.
		
		\medskip

		\noindent
		\textit{Special case 1: Real mobility tensor.}
		Assume $\mathbf{B} = \mathbf{B}_R$, where $\mathbf{B}_R$ is real and symmetric. Then, comparing with \eqref{Expression_Theta}, we let $\Theta_{R}$ denote the contribution to $\Theta$ in this case.
		We define
		\[
		z = i \omega \rho_f \, (\mathbf{B}_R \cdot \mathbf{e}_1^{k+1}) 
		\cdot \overline{\nabla e_2^{k+1}}.
		\]
		Then we obtain
		\[
		\Theta_R = \int_\Omega (z - \bar{z})\, d\mathbf{x}=2i \int_\Omega \operatorname{Im}(z)\, d\mathbf{x}.
		\]
		Hence, $\Theta_R$ is purely imaginary, and therefore
		\[
		\mathrm{Re}(\Theta_R) = 0.
		\]

		\noindent
		\textit{Special case 2: Purely imaginary mobility tensor.}
		To expose the structure of $\Theta$ in the purely imaginary mobility case, we introduce the Hermitian tensor
		\[
		\mathbf{H}
		:=
		-i\omega \rho_f \Big(
		(\nabla e^{k+1}_2)\otimes \overline{\mathbf{e}^{k+1}_1}
		- \mathbf{e}^{k+1}_1 \otimes \overline{\nabla e^{k+1}_2}
		\Big),
		\]
		so that
		\begin{equation}
			\Theta = \int_\Omega \mathbf{H} : \mathbf{B} \, dx,
			\label{Theta_representation}
		\end{equation}
		where $:$ denotes the double contraction.
		One readily verifies that $\mathbf{H}$ is Hermitian, i.e. $\mathbf{H}^\dagger = \mathbf{H}$.
		Assume $\mathbf{B} = i\mathbf{B}_{\mathrm{I}}$, where $\mathbf{B}_{\mathrm{I}}$ is real and symmetric. 
		Then, using the representation \eqref{Theta_representation}, we obtain
		\[
		\Theta_{\mathrm{I}} = \int_\Omega \mathbf{H} : (i\mathbf{B}_{\mathrm{I}}) \, dx=\int_\Omega (i\mathbf{H}) : \mathbf{B}_{\mathrm{I}} \, dx=\int_\Omega \mathbf{A} : \mathbf{B}_{\mathrm{I}} \, dx, \quad \mbox{where }\mathbf{A} := i \mathbf{H}.
		\]
		This identity shows that $\Theta_{\mathrm{I}}$ denotes the contribution to $\Theta$ corresponding to a purely imaginary mobility tensor.
		Since $\mathbf{H}$ is Hermitian, $\mathbf{A}$ is anti-Hermitian. The contraction of an anti-Hermitian tensor with a real symmetric tensor is purely imaginary. Hence,
		\[
		\mathrm{Re}(\Theta_{\mathrm{I}}) = 0.
		\]
		
		\medskip
		
		\noindent
		\textit{General case: complex mobility tensor}.
		Let
		\[
		\mathbf{B} = \mathbf{B}_{\mathrm{R}} + i \mathbf{B}_{\mathrm{I}}.
		\]
		Recall from special cases 1 and 2 that $\Theta_R$ and $\Theta_I$ denote the contributions to $\Theta$ associated with this decomposition, and not the real and imaginary parts of $\Theta$.
		Using the linearity of the mapping $\mathbf{B} \mapsto \Theta$ given in \eqref{Theta_representation}, we obtain
		\[
		\Theta = \int_\Omega \mathbf{H} : \mathbf{B}_{\mathrm{R}} \, dx
		+ \int_\Omega \mathbf{H} : (i \mathbf{B}_{\mathrm{I}}) \, dx
		= \Theta_{\mathrm{R}} + \Theta_{\mathrm{I}}.
		\]
		Since both $\Theta_{\mathrm{R}}$ and $\Theta_{\mathrm{I}}$ are purely imaginary, it follows that
		\[
		\mathrm{Re}(\Theta) = 0.
		\]
		This shows that the instantaneous coupling term does not contribute to the real part of the error energy.

		\noindent\textbf{Step 4 (Assemble the inequality and bound the difference term).}\\
		\color{black}
		By applying the binomial identity in Lemma \color{black} 2 and by combining \eqref{eq:energyline} and \eqref{eq:Msplit} we obtain
		\begin{align*}
			\alpha_{11}\|\bm{e}_{1}^{k+1}\|^2 + \alpha_{22}\|e_2^{k+1}\|^2
			&+ \tfrac{L}{2}\Big(\|\bm{e}_{1}^{k+1}\|^2 + \|\bm{e}_{1}^{k+1}-\bm{e}_{1}^{k}\|^2 - \|\bm{e}_{1}^{k}\|^2\Big) \\
			&\quad + \Ree\,\ip{\mathcal{A}_{12}(e_2^{k}-e_2^{k+1})}{\bm{e}_{1}^{k+1}}\leq 0.
		\end{align*}
		
		Using the Cauchy--Schwarz inequality and boundedness of $\mathcal{A}_{12}$,
		\[
		|\Ree\,\ip{\mathcal{A}_{12}(e_2^{k}-e_2^{k+1})}{\bm{e}_{1}^{k+1}}|
		\le \|\mathcal{A}_{12}(e_2^{k}-e_2^{k+1})\|\,\|\bm{e}_{1}^{k+1}\|
		\le \|\mathcal{A}_{12}\|\,\|e_2^{k}-e_2^{k+1}\|\,\|\bm{e}_{1}^{k+1}\|.
		\]
		Using the difference estimate in Lemma 3,
		\[
		\|e_2^{k}-e_2^{k+1}\|
		\le \frac{\|\mathcal{A}_{21}\|}{\alpha_{22}}\,\|\bm{e}_{1}^{k+1}-\bm{e}_{1}^{k}\|,
		\]
		we obtain
		\[
		|\Ree\,\ip{\mathcal{A}_{12}(e_2^{k}-e_2^{k+1})}{\bm{e}_{1}^{k+1}}|
		\le \frac{\|\mathcal{A}_{12}\|\,\|\mathcal{A}_{21}\|}{\alpha_{22}}\,
		\|\bm{e}_{1}^{k+1}-\bm{e}_{1}^{k}\|\,\|\bm{e}_{1}^{k+1}\|.
		\]
		Applying Young's inequality with the choice $\tau=\alpha_{11}$,
		gives the bound
		\begin{equation*}
			\Ree\,\ip{\mathcal{A}_{12}(e_2^{k}-e_2^{k+1})}{\bm{e}_{1}^{k+1}}
			\le \frac{\alpha_{11}}{2}\,\|\bm{e}_{1}^{k+1}\|^2
			+ \frac{\|\mathcal{A}_{12}\|^2 \|\mathcal{A}_{21}\|^2}{2\alpha_{11}\alpha_{22}^2}\,
			\|\bm{e}_{1}^{k+1}-\bm{e}_{1}^{k}\|^2.
		\end{equation*}

		\noindent\textbf{Step 5 (Inequality with condition on $L$).}\\ 
		
		Inserting the bounds from Step 4 and absorbing like terms to the left-hand side, we get
		after multiplying by $2$,
		\begin{equation}\label{eq:inequality}
			(L+\alpha_{11}) \|\bm{e}_{1}^{k+1}\|^2
			+ \underbrace{\big(L - L_{\min}\big)}_{=:c_1\geq 0}\,\|\bm{e}_{1}^{k+1}-\bm{e}_{1}^{k}\|^2
			+ 2\alpha_{22}\,\|e_2^{k+1}\|^2
			\;\le\; L \|\bm{e}_{1}^{k}\|^2.
		\end{equation}
		Here $L_{\min}$ is defined in \eqref{eq: Lmin}. Under the stated conditions, $c_1>0$.
		From the inequality in \eqref{eq:inequality}, ignoring the nonnegative terms 
		$c_1\|\bm{e}_{1}^{k+1}-\bm{e}_{1}^{k}\|^2$ and $2\alpha_{22}\|e_2^{k+1}\|^2$, we obtain
		\[
		\|\bm{e}_{1}^{k+1}\|^2 \le \frac{L}{L+\alpha_{11}} \|\bm{e}_{1}^{k}\|^2.
		\]
		Thus, the iteration is a contraction with factor $\rho$ defined in \Cref{thm: main theorem}
		which shows that the method converges linearly. 

	\end{proof}
	
	\begin{remark}[Physical Interpretation]
		The cancellation of the instantaneous coupling terms reflects that solid–fluid interaction does not exchange real energy instantaneously. Dissipative effects arise only through memory terms encoded in the Hermitian part $\operatorname{Im}(\mathbf{B})$, while the real part $\Ree(\mathbf{B})$ contributes only to reversible processes. Thus, the instantaneous coupling acts as a purely reactive term, preserving energy stability.
		
	\end{remark}
	
	
	\FloatBarrier

	\section{Numerical examples}\label{sec: numerics}
	In this section, we perform three numerical experiments to study the robustness and performance of the splitting scheme. First, we consider a homogeneous medium under different parameter regimes. We test the performance of the scheme by considering different stabilization parameters. In the second example, we consider a three-layered heterogeneous domain, illustrated in \Cref{fig: hetrougenous domain}, and study the number of iterations required for convergence. The first two examples are 2-dimensional and discretized using the finite element method, specifically, the Taylor-Hood elements. Here, all linear systems are solved using a direct solver.  The last example is a 3-dimensional problem in which we consider a heterogeneous, vertically transversely isotropic medium. For this example, we use the pseudospectral method for discretization, and the linear system arising in the splitting scheme is solved using GMRES. 
	\begin{figure}
		\centering
		\includegraphics[width=0.48\linewidth]
		{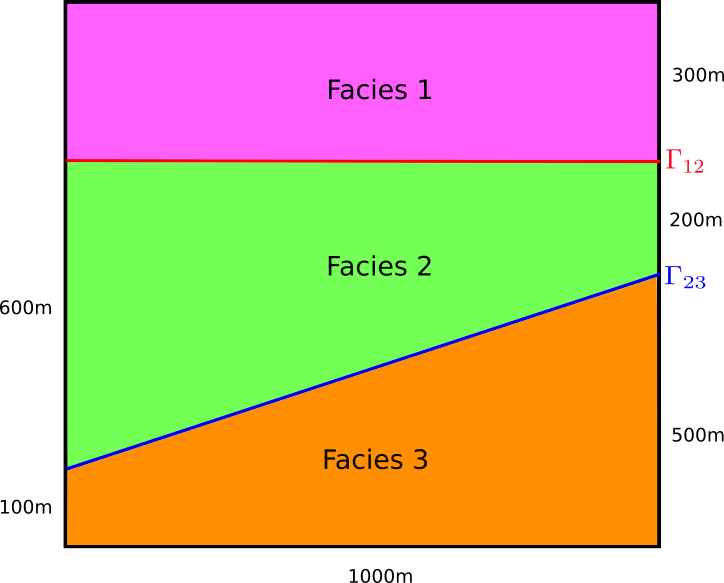}
		\caption{Test case 2 - Illustration of the heterogeneous domain.}
		\label{fig: hetrougenous domain}
	\end{figure}
	
	All examples will use a Ricker wavelet to mimic a seismic source, 
	\begin{equation}\label{eq: ricker}
		f(t)=(1 - 2 (\pi  f_0  (t - t_0))^2)  e^{-(\pi f_0  (t - t_0))^2},
	\end{equation}
	where $t_0=1/f_0$ with $f_0$ being the peak frequency. An example of the Ricker wavelet with peak frequency $f_0=20$Hz is displayed in \Cref{fig: ricker wavelet} along with a few selected frequencies where we look at the performance of the iterative scheme for the first two examples.  For the dynamic permeability, we will restrict ourselves to only consider the JKD dynamic permeability \cite{Johnson1987} 
	\begin{equation}\label{eq: JKD dynamic permeability}
		\boldsymbol{\kappa}(\omega) = \frac{k_0}{\sqrt{1-\frac{4i\alpha_{\infty}^{2}k_0^{2}\rho_f\omega}{\Lambda^{2}\eta \phi^{2}}}-\frac{i\alpha_{\infty}k_0\rho_f\omega}{\eta \phi}   },
	\end{equation}
	where $a_{\infty}$ is the infinite-frequency tortuosity and $\Lambda$ a tunable geometric parameter. To compute the tunable geometry constant $\Lambda$, we use, for simplicity, the empirical relation \cite[Eq. (7.278)]{Carcione2022}. 
	\begin{figure}\centering\includegraphics[width=0.7\textwidth]{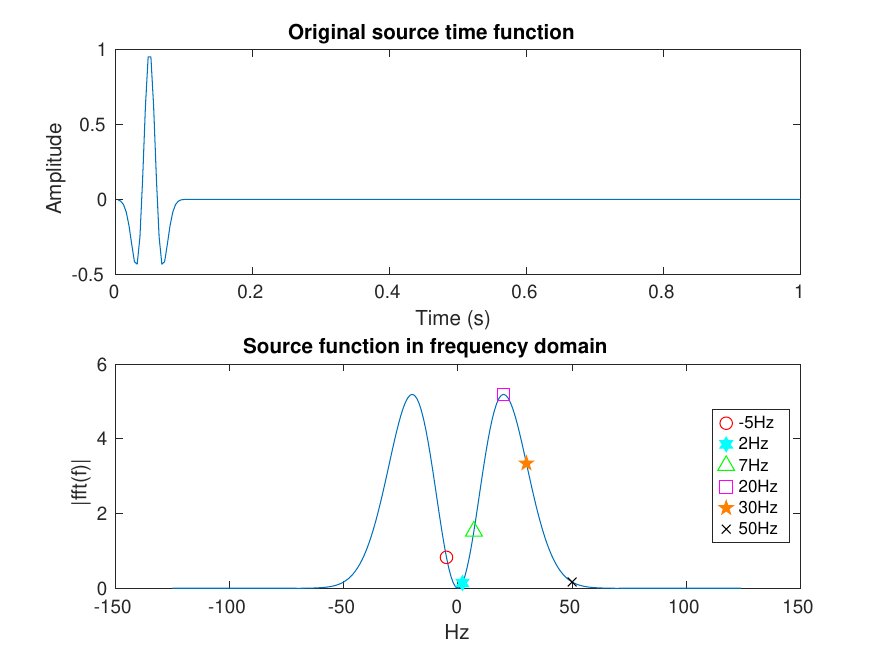}
		\caption{Ricker wavelet with a peak frequency of $20$Hz in the time and frequency domains, with marked frequencies, where we show the performance of the iterative scheme for Test case 1 and 2.}
		\label{fig: ricker wavelet}
	\end{figure}
	
	To measure the performance of the scheme, we look at the number of iterations required for different stabilization parameters across various medium parameters. For all examples, the iterative procedure is stopped when the relative error satisfies
	\begin{align}
		\|(\bmv^{k},p^{k})-(\bmv^{k-1},p^{k-1})\|\leq 10^{-6}\|(\bmv^{k},p^{k})\|.
	\end{align}
	The implementation is done in FreeFem++ \cite{freefem} for the 2-dimensional examples and in Matlab for the 3-dimensional example.

	\subsection{Test case 1 - Homogeneous isotropic 2D model}
	We consider a square homogeneous medium of size 1km by 1km on a structured triangular mesh with a mesh diameter of $5.657$m.  We will include a mechanical source term of the form 
	\begin{equation}\label{eq: source term examples}
		\textbf{f}=\textbf{b}f(t)\delta(\textbf{x}-\textbf{x}_s),
	\end{equation}
	where $\textbf{x}_s$ is the source location, and $\delta(\textbf{x}-\textbf{x}_s)$ is a regularized Dirac delta, i.e. a point Gaussian with standard deviation equal to 1.5$h_{loc}$ with $h_{loc}$ being the maximum mesh size of the elements touching $\bm{x}_s$, located at $\textbf{x}_s=(500,500)$, $f(t)$ is the time impulse function and $\textbf{b}=[0,1]^{T}$ is a vectorial function determining the strength of the source in different directions. 
		\begin{figure}
		\centering
		\begin{subfigure}[t]{0.49\textwidth}
			\centering
			\includegraphics[width=\linewidth]{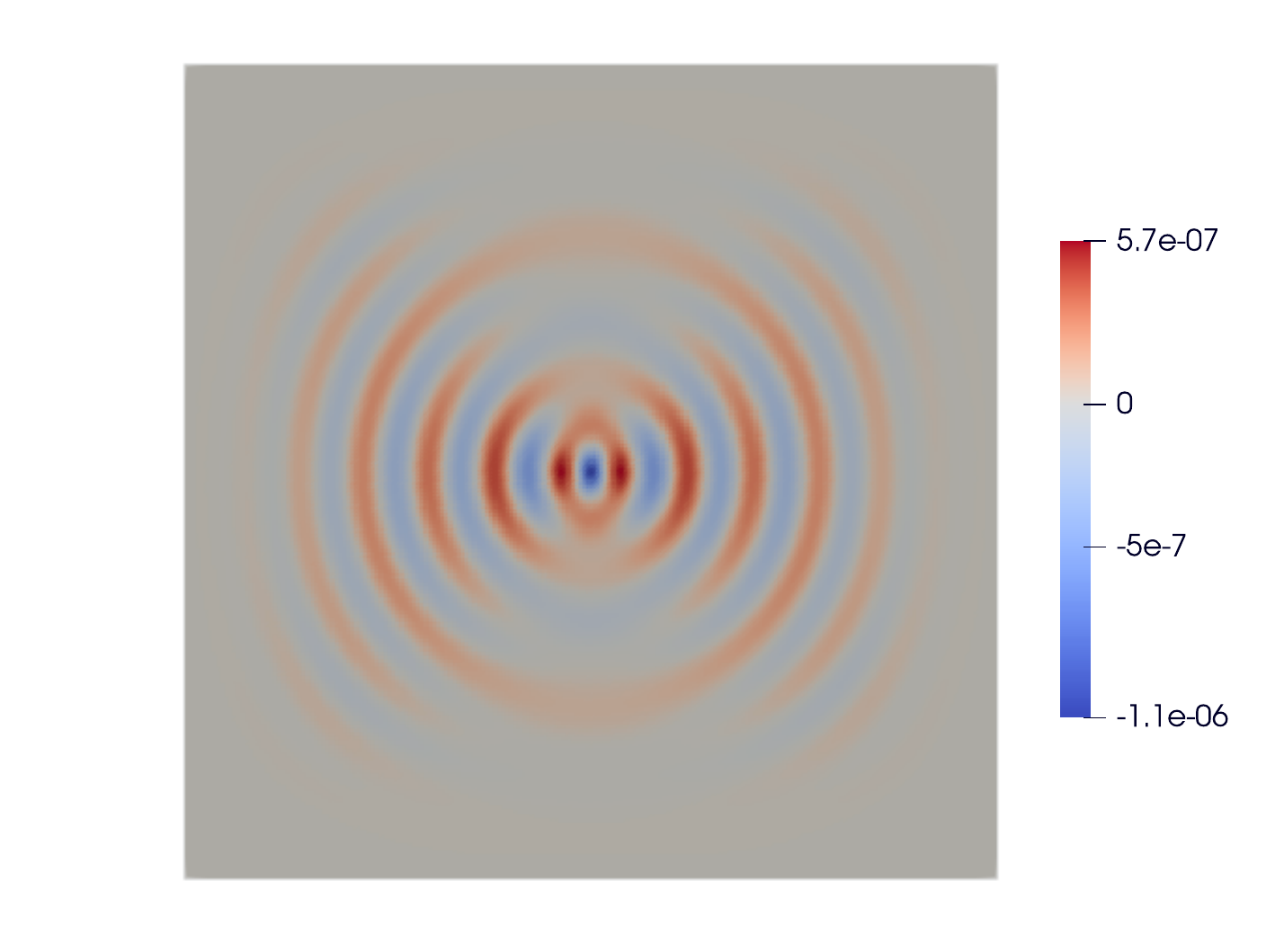}
			\caption{Real part}
		\end{subfigure}
		\hfill
		\begin{subfigure}[t]{0.49\textwidth}
			\centering
			\includegraphics[width=\linewidth]{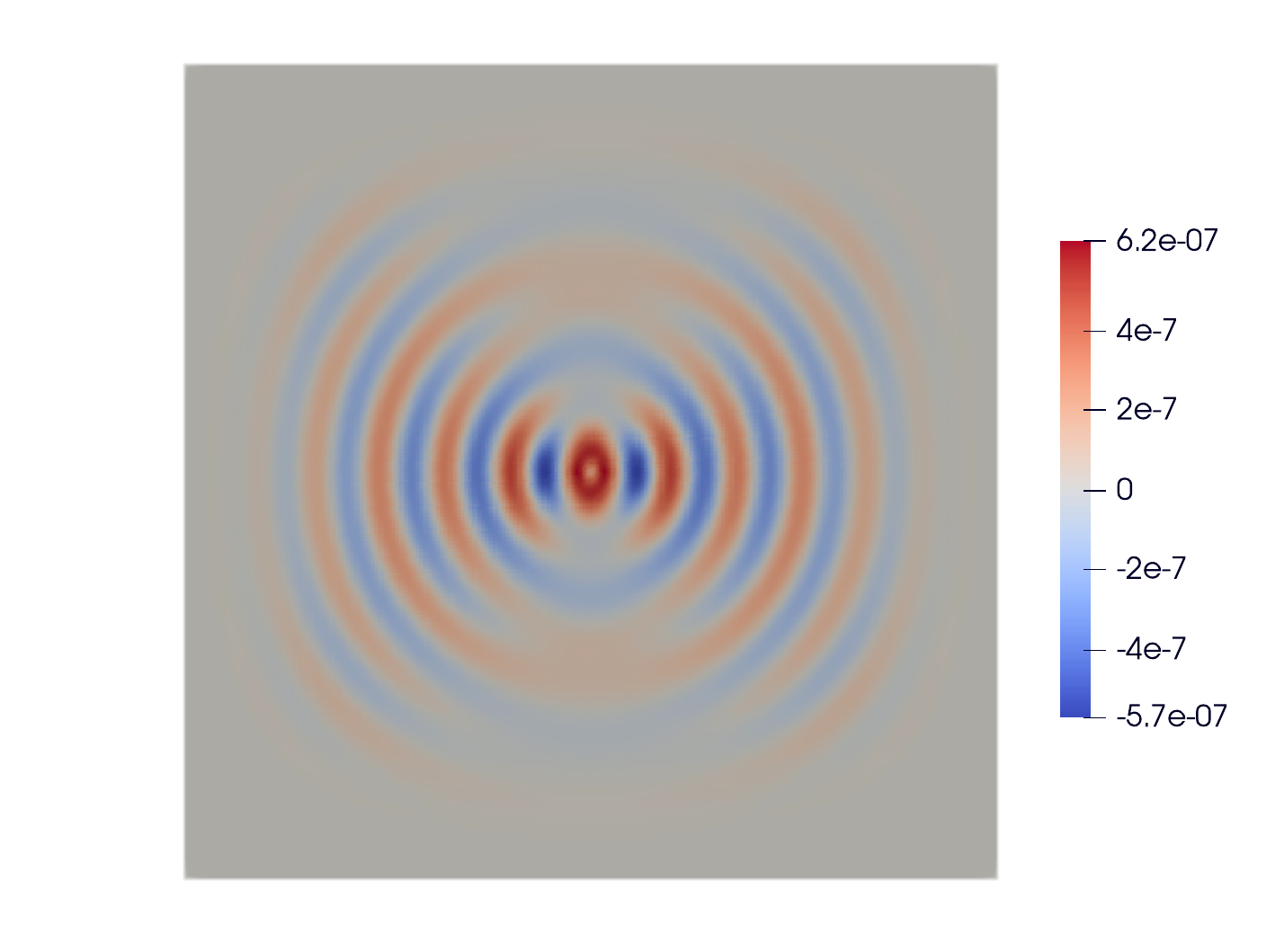}
			\caption{Imaginary part}
		\end{subfigure}
		
		\caption{Test case 1 - Numerical solution of $\textbf{v}_y$ at 30Hz for $\alpha=0.189$, $v_p=5800m/s$, and $v_s=2410m/s$ with a mesh diameter of $5.657$m.}
		\label{fig: example 1 velocity y comp solution}
	\end{figure}
		\begin{table}
		\renewcommand{\arraystretch}{1.1} 
		\centering
		\caption{Test case 1 - Poroelastic Constants for the Homogeneous Medium}
		\begin{tabular}{ll}
			\hline
			\textbf{Parameter} [unit] & \textbf{Value} \\ \hline
			P-wave velocity (\(v_p\)) [m/s] & \{ 5800, 4800 \} \\ 
			S-wave velocity (\(v_s\)) [m/s]& \{ 2410, 1410\}\\ 
			Density (\(\rho\)) [kg/m\(^3\)] & 2370  \\ 
			Fluid density (\(\rho_f\)) [kg/m\(^3\)]& 1000  \\ 
			Porosity (\(\phi\)) & 0.16 \\ 
			Permeability (\(k\)) [m\(^2\) ] & \(9.869233 \times 10^{-13}\) \\ 
			Viscosity (\(\eta\)) [Pa·s ] & \(1 \times 10^{-3}\) \\ 
			Bulk modulus of the solid matrix (\(K_s\)) [GPa]& 37  \\ 
			Bulk modulus of the fluid (\(K_f\)) [GPa] & 2.2  \\ 
			Shear modulus (\(\mu\)) [GPa]& \{13.8, 4.7  \}\\
			Effective stress coefficient (\(\alpha\)) & \{ 0.189, 0.5, 1 \}  \\ 
			Central frequency ($f_0$)& 20\\  
			Infinite-frequency tortuosity ($\alpha_{\infty}$) & 1.5 \\ \hline 
		\end{tabular}
		\label{tab:poroelastic_constants}
	\end{table}
	
	We apply homogeneous boundary conditions on the boundary of $\Omega$. However, in the region near the boundaries with a thickness $d_{AB}$, we use coordinate stretching to mimic absorbing boundary conditions. In essence, we transform the partial derivatives $\partial /\partial x\mapsto 1/s_x\partial /\partial x$ where $s_x$ is a stretching function depending on space and frequency. We employ a simple stretching function of the form $1+i\textbf{D}_x/\omega$ where the damping function $\textbf{D}_x=((d_{AB}-x)/d_{AB})^{2}$ grows quadratically with the distance into the region near the boundary. We refer the reader to an example of this methodology applied to Maxwell's equations in \cite{chew1997}; it is also commonly used in the context of dynamic poroelastic systems \cite{he2019}. The region near the boundary where we apply coordinate stretching has a width of $d_{AB}=200$m.
	We consider different parameters for the homogeneous medium, which are displayed in \Cref{tab:poroelastic_constants}. For $\omega=30$Hz, the numerical solution is displayed in \Cref{fig: example 1 velocity y comp solution}, \Cref{fig: example 1 velocity x comp solution}, and \Cref{fig: example 1 pressure solution} when $\alpha=0.189$, p-wave velocity is 5800m/s, and s-wave velocity is 2410m/s. Since the medium is assumed to be isotropic, the effective stress tensor is given by $\alpha _{ij} = \alpha \delta _{ij}$, with a slight abuse of notation.



	\begin{figure}
		\centering
		\begin{subfigure}[t]{0.49\textwidth}
			\centering
			\includegraphics[width=\linewidth]{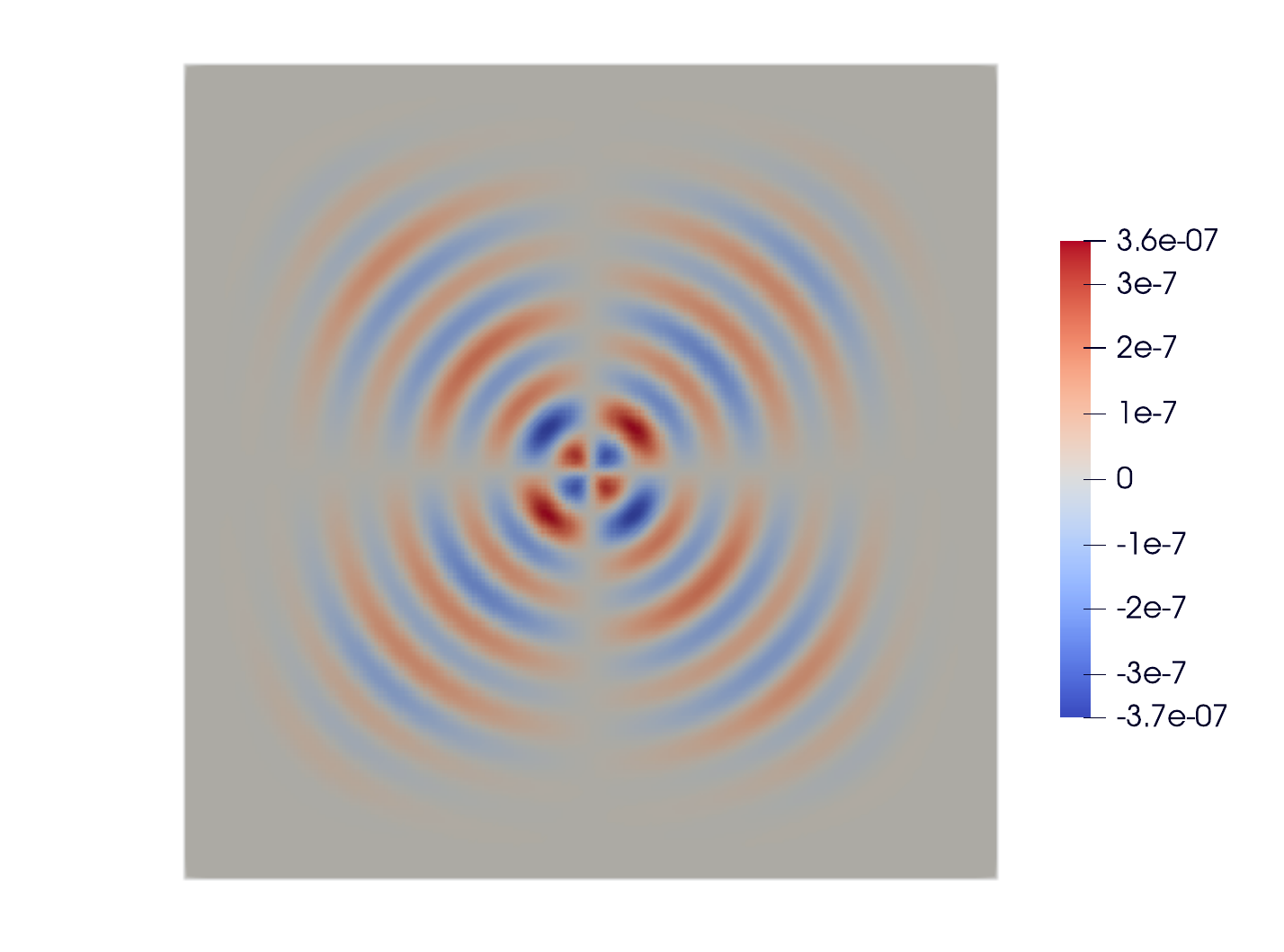}
			\caption{Real part}
		\end{subfigure}
		\hfill
		\begin{subfigure}[t]{0.49\textwidth}
			\centering
			\includegraphics[width=\linewidth]{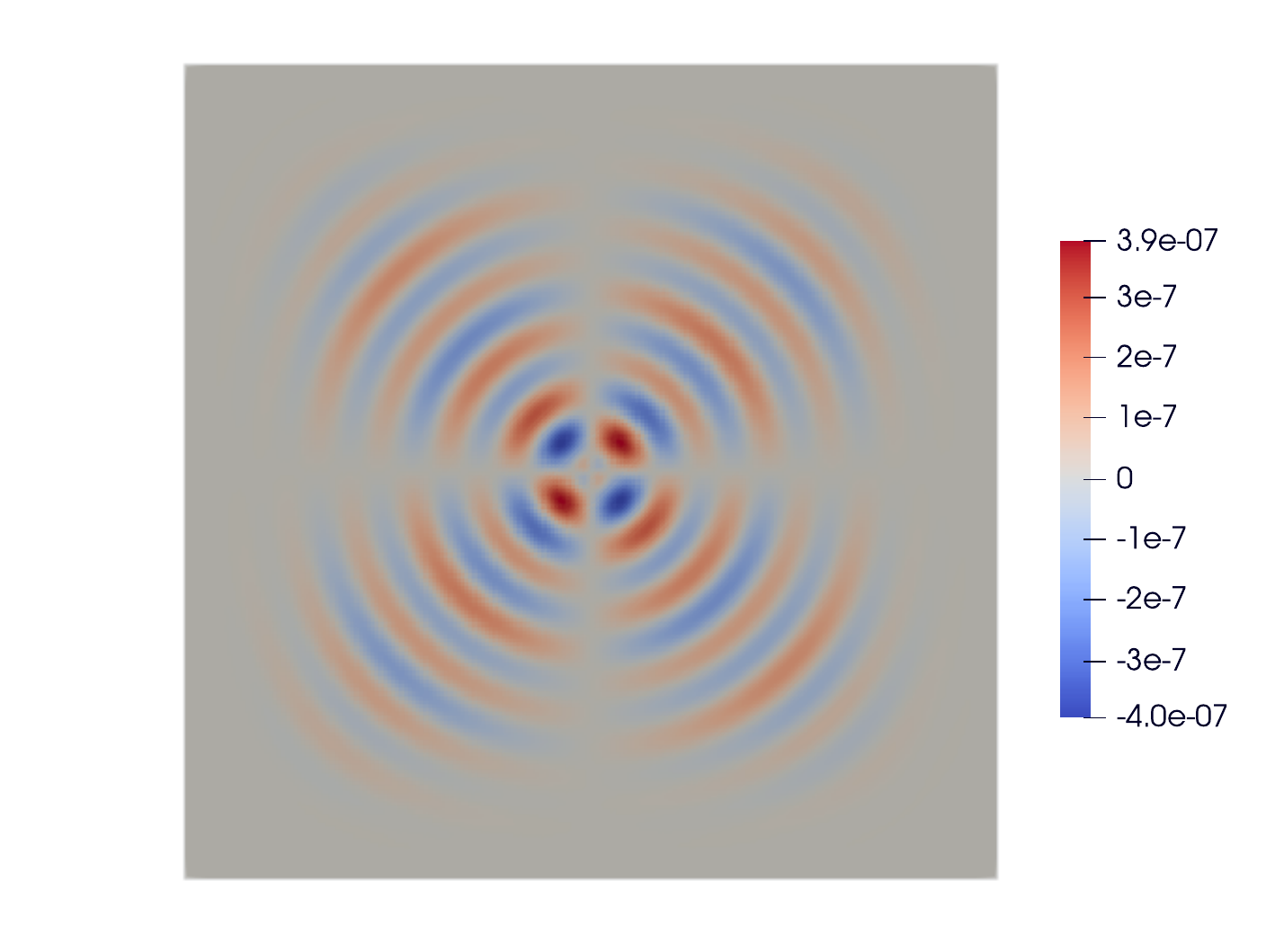}
			\caption{Imaginary part}
		\end{subfigure}
		
		\caption{Test case 1 - Numerical solution of $\textbf{v}_x$ at 30Hz for $\alpha=0.189$, $v_p=5800m/s$, and $v_s=2410m/s$ with a mesh diameter of $5.657$m.}
		\label{fig: example 1 velocity x comp solution}
	\end{figure}
	
	\begin{figure}
		\centering
		\begin{subfigure}[t]{0.49\textwidth}
			\centering
			\includegraphics[width=\linewidth]{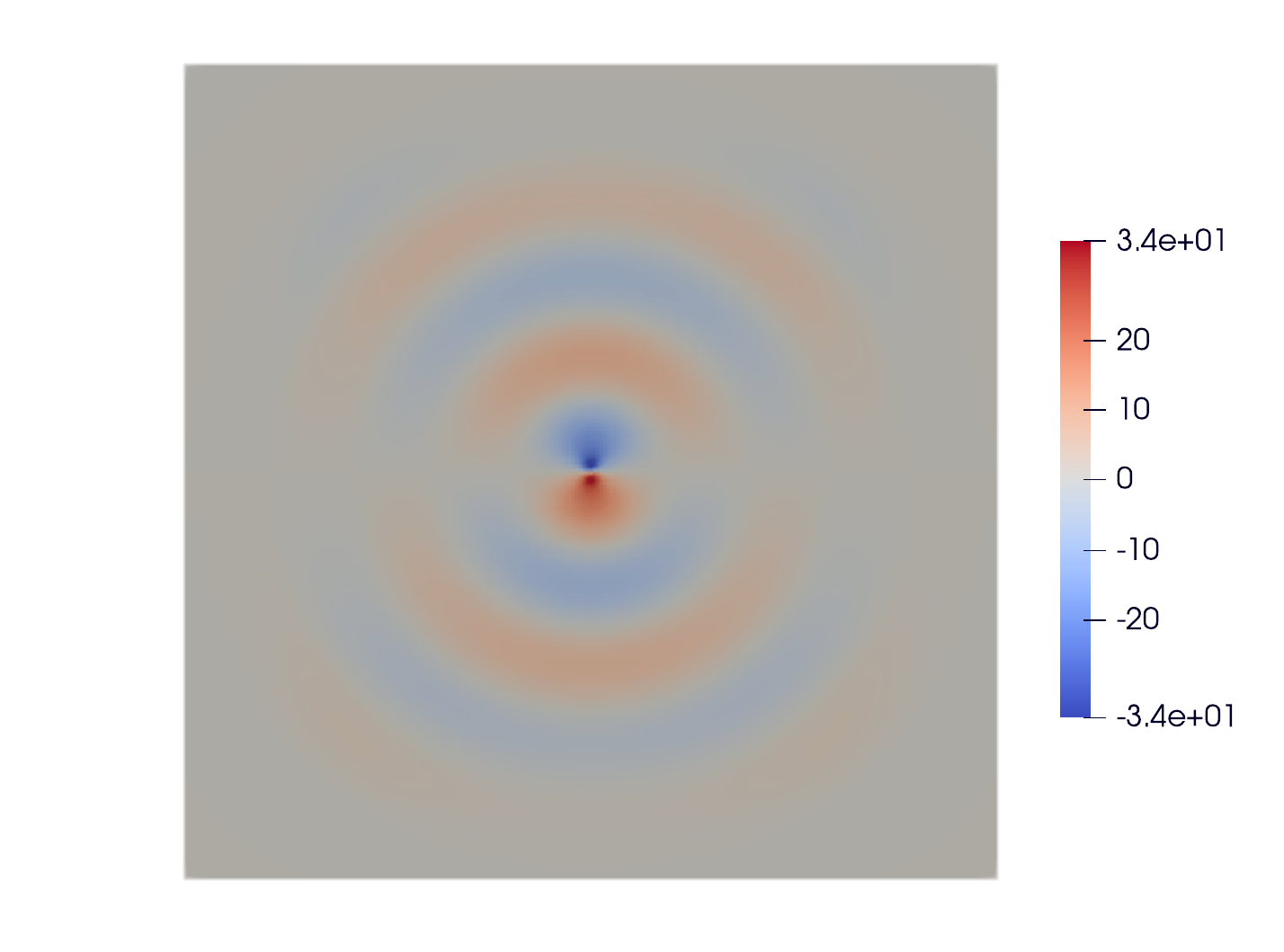}
			\caption{Real part}
		\end{subfigure}
		\hfill
		\begin{subfigure}[t]{0.49\textwidth}
			\centering
			\includegraphics[width=\linewidth]{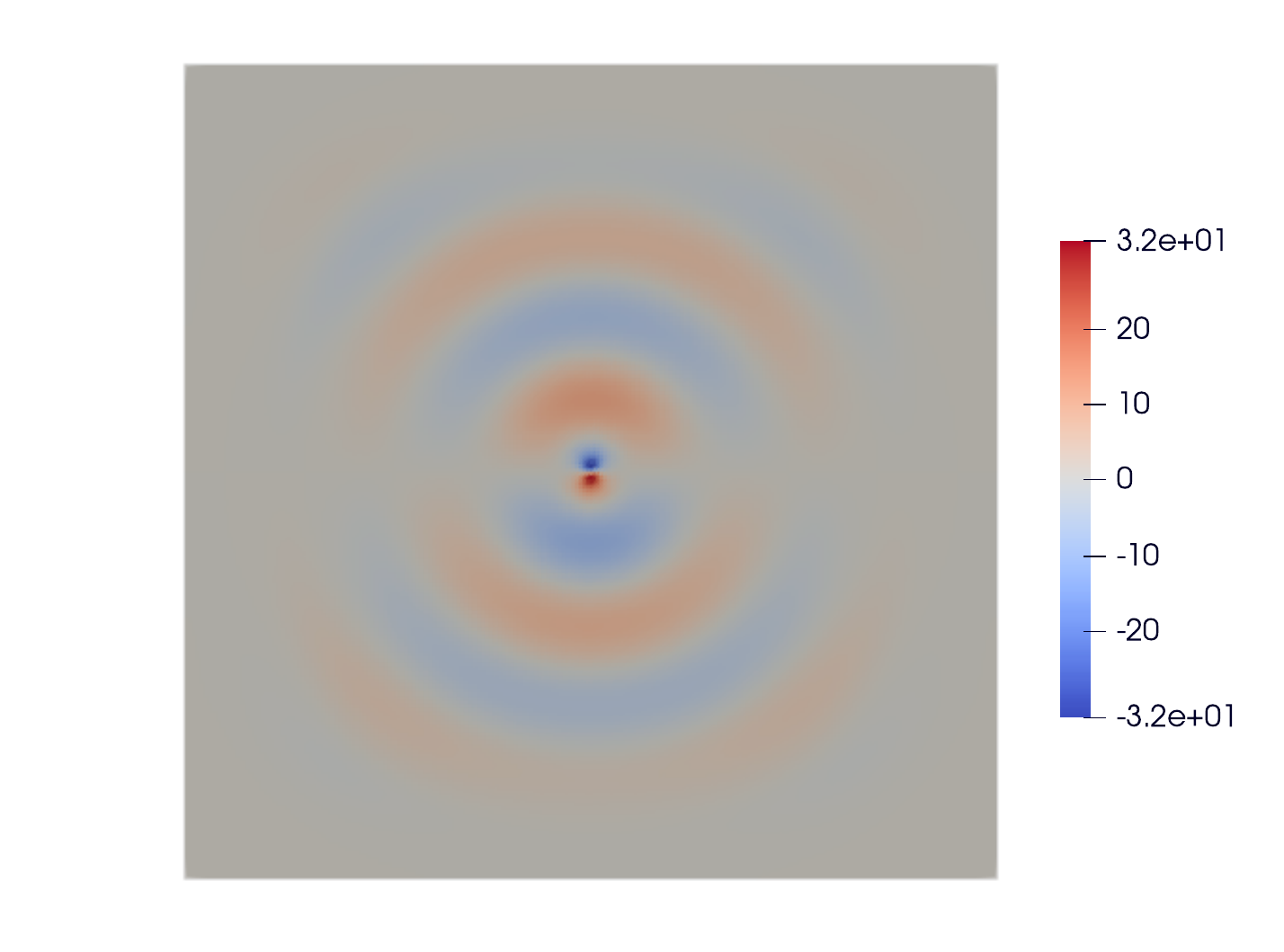}
			\caption{Imaginary part}
		\end{subfigure}
		
		\caption{Test case 1 - Numerical solution of $\dot{p}$ at 30Hz for $\alpha=0.189$, $v_p=5800m/s$, and $v_s=2410m/s$ with a mesh diameter of $5.657$m.}
		\label{fig: example 1 pressure solution}
	\end{figure}

	In \Cref{fig: example 1 small alpha} and \Cref{fig: example 1 varying alpha}, the number of iterations for different stabilization parameters under different material parameters is shown. In the case where $\alpha$ is small, the number of iterations is smaller than for larger values of $\alpha$. In addition, the number of iterations varies less for different stabilization parameters. For larger $\alpha$, the number of iterations appears to vary more for different stabilization parameters, with $\alpha^{2}/\beta$ yielding the fewest number of iterations. For worse choices of stabilization parameters, we also observe that the problem is frequency-dependent, with higher frequencies resulting in more iterations. For lower wave speeds, the number of iterations slightly increases.  This trend is further exemplified in \Cref{fig: example 1 wavespeeds} where we see divergence of the iterative scheme. This is expected since, when lowering the wave speed, we eventually violate the coercivity assumption of the elastic part of the system on which the convergence proof is based. Consequently, the iterative scheme is not expected to converge. In \Cref{sec: appendix A}, the coercivity of the diagonal operators is further discussed.
	\begin{figure}
		\centering
		\includegraphics[width=0.58\linewidth]{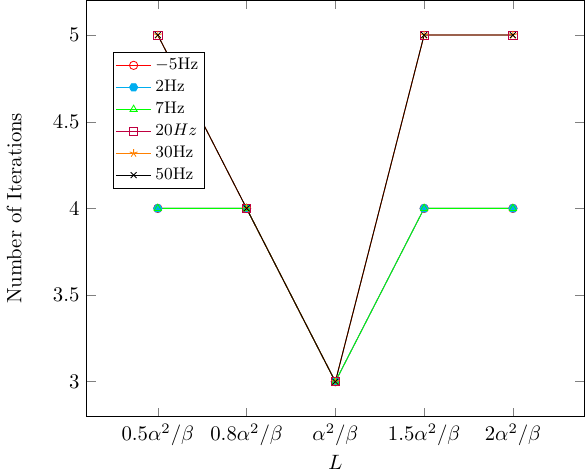}
		\caption{Test case 1 - Number of iterations for different frequencies with different stabilization parameters and $\alpha=0.189$, $v_p=5800m/s$, and $v_s=2410m/s$ with a mesh diameter of $5.657$m.}
		\label{fig: example 1 small alpha}
	\end{figure}

	\begin{figure}
		\centering
		\begin{subfigure}[t]{0.47\textwidth}
			\centering
			\includegraphics[width=0.95\linewidth]{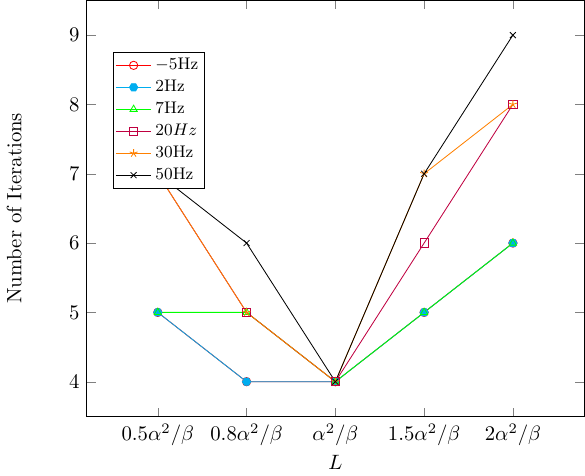}
			\caption{$\alpha=0.5$}
		\end{subfigure}
		\hfill
		\begin{subfigure}[t]{0.47\textwidth}
			\centering
			\includegraphics[width=0.95\linewidth]{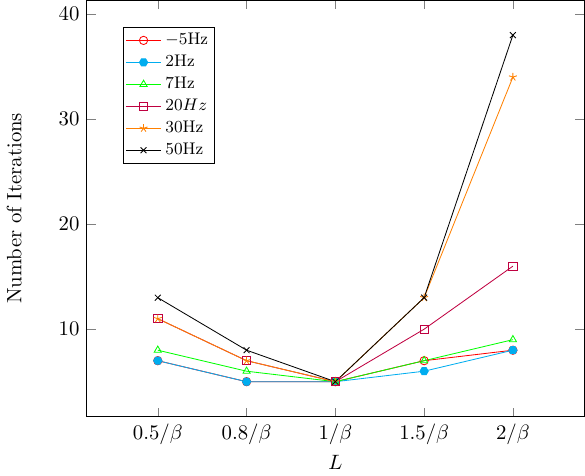}
			\caption{$\alpha=1$}
		\end{subfigure}
		
		\vspace{1em} 
		
		\begin{subfigure}[t]{0.47\textwidth}
			\centering
			\includegraphics[width=0.95\linewidth]{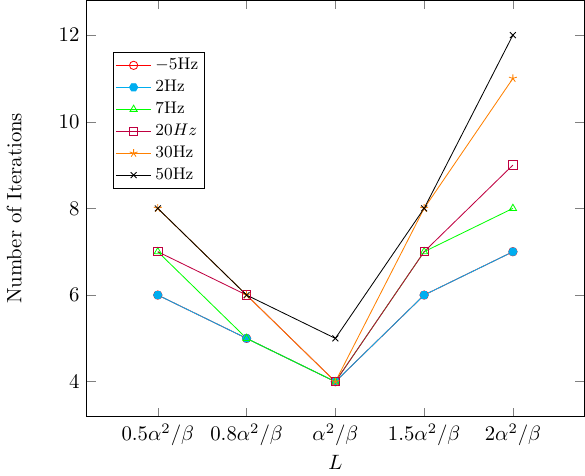}
			\caption{$\alpha=0.5$}
		\end{subfigure}
		\hfill
		\begin{subfigure}[t]{0.47\textwidth}
			\centering
			\includegraphics[width=0.95\linewidth]{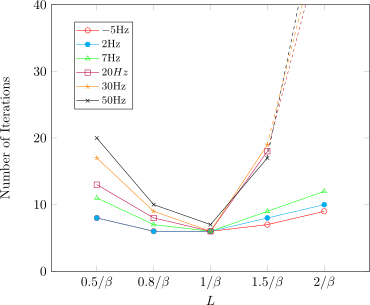}
			\caption{$\alpha=1$}
		\end{subfigure}
		
		\caption{Test case 1 - Number of iterations for different frequencies with different stabilization parameters with a mesh diameter of $5.657$m. In (a) and (b) $P$- and $S$- wave velocity is 5800m/s and 2410 m/s and $\mu=13.8$GPa. In (c) and (d) $P$- and $S$- wave velocity is 4800m/s and 1410 m/s and $\mu=4.7$GPa. Dashed lines indicate convergence achieved after 80 iterations. }
		\label{fig: example 1 varying alpha}
	\end{figure}
	
	\begin{figure}
		\centering
		\includegraphics[width=0.6\linewidth]{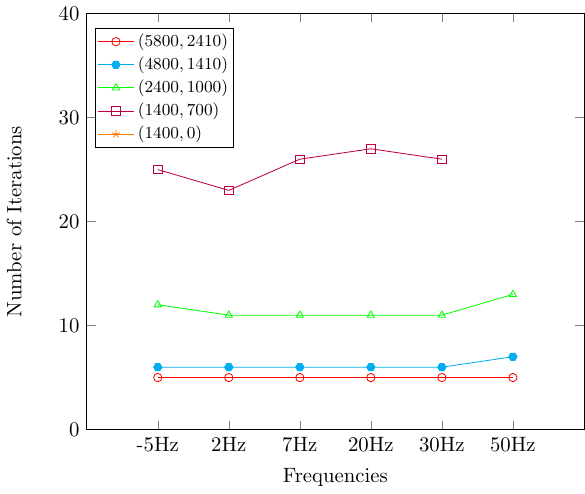}
		\caption{Test case 1 - Number of iterations for $\alpha=1$ with $L=\beta^{-1}$ for different $v_p$ and $v_s$ values  with a mesh diameter of $5.657$m. Missing values imply divergence.}
		\label{fig: example 1 wavespeeds}
	\end{figure}
	
	
	\subsection{Test case 2 - Heterogeneous isotropic 2D model}

	We consider a square layered heterogeneous medium of size 1km by 1km, see \Cref{fig: hetrougenous domain}. The mesh is discretized into 130118 triangles with a minimum mesh size of $2.5293 m$ and a maximum of $8.00094 m$. Along interface $\Gamma_{12}$ we have 280 elements and along $\Gamma_{23}$ we have 300 elements. The source \eqref{eq: source term examples} is centered in the middle of the domain, i.e. $\textbf{x}_s=(500,500)$ and acts in the $y$-direction, i.e. $\textbf{b}=[0,1]^{T}$. We consider a fluid density \(\rho_f\) = 1000 kg/m\(^3\), viscosity \(\eta\) = \(1 \times 10^{-3}\) Pa·s and a bulk modulus of the fluid \(K_f\) = $2.2$ GPa. We again apply homogeneous boundary conditions with an absorbing region $d_{AB}=200m$. Since the medium is heterogeneous, we consider a spatially dependent stabilization parameter. By considering only one region or the worst-case scenario, the convergence is either very slow or the scheme diverges. The y-component of the velocity at 30Hz is displayed in \Cref{fig: example 2 velocity y comp solution}. As expected, we see that the wavelength is longer in Facies 1 than in Facies 3 due to the difference in wave speed.

	\begin{table}[ht]
		\renewcommand{\arraystretch}{1.1} 
		\caption{Test case 2 - Poroelastic constants of the heterogeneous medium}
		\begin{tabular}{llll}
			\hline
			\textbf{Parameter} [unit]& \textbf{Facies 1} & \textbf{Facies 2}  & \textbf{Facies 3}\\ \hline
			P-wave velocity (\(v_p\)) [m/s] & 3200  & 2400 & 1900\\ 
			S-wave velocity (\(v_s\)) [m/s]& 1900 & 1400 & 1200\\ 
			Solid density (\(\rho_s\)) [kg/m\(^3\)]& 2650 & 2600 & 2550  \\ 
			Porosity (\(\phi\)) & 0.15 & 0.3 & 0.35\\ 
			Permeability (\(k\)) [m\(^2\)] & \(1 \times 10^{-12}\) &\(1 \times 10^{-11}\) &\(5 \times 10^{-11}\)  \\ 
			Bulk modulus of the solid matrix (\(K_s\)) [GPa]& 15  & 20 & 40\\
			Shear modulus (\(\mu\)) [GPa] & 8.67 & 4.16 & 2.89\\ 
			Effective stress coefficient (\(\alpha\)) & 0.073 &0.2&0.75\\ \hline
		\end{tabular}
		\label{tab:poroelastic_constants2}
	\end{table}
	
	\begin{figure}[H]
		\centering
		\begin{subfigure}[t]{0.49\textwidth}
			\centering
			\includegraphics[width=\linewidth]{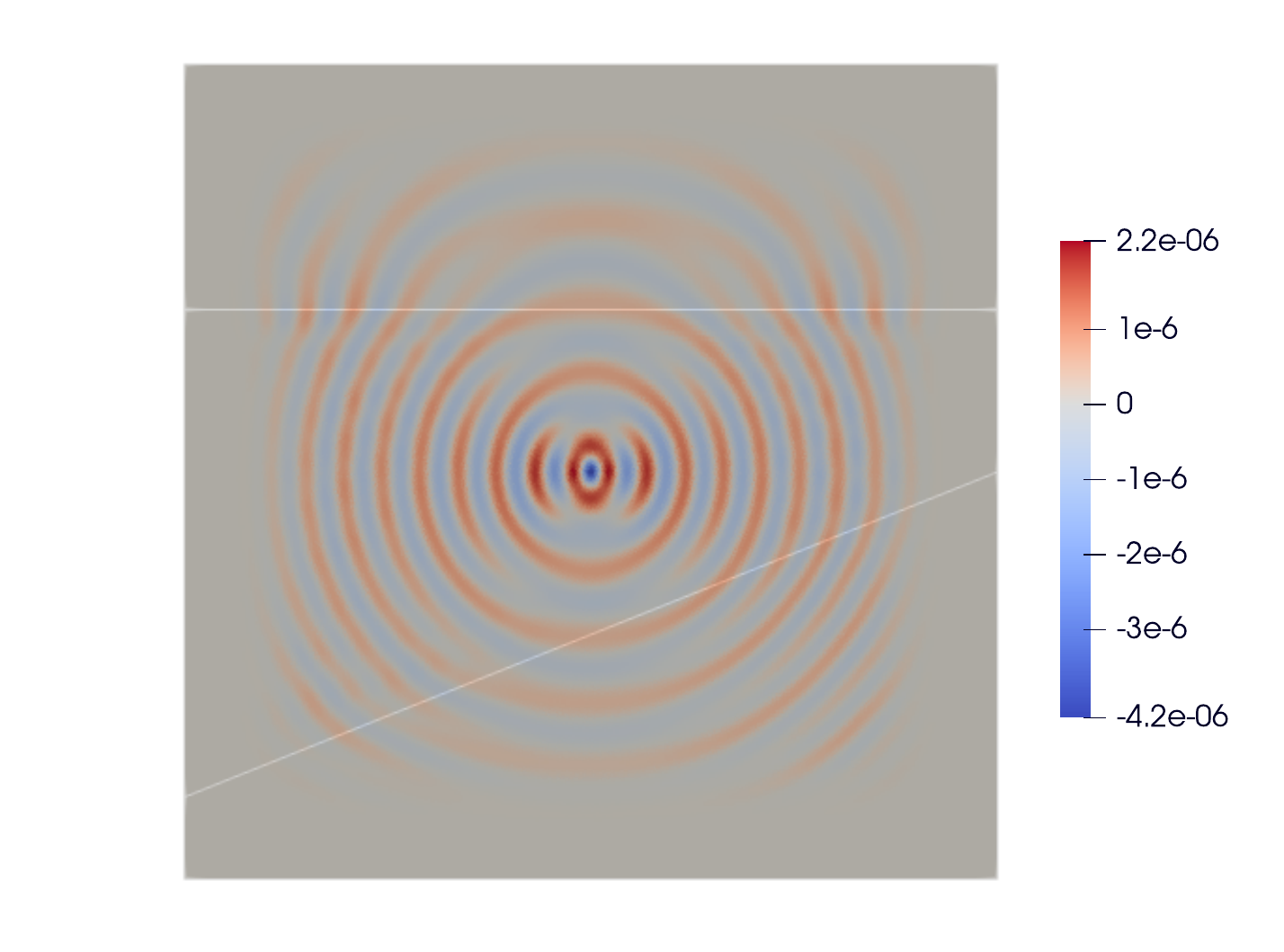}
			\caption{Real part}
		\end{subfigure}
		\hfill
		\begin{subfigure}[t]{0.49\textwidth}
			\centering
			\includegraphics[width=\linewidth]{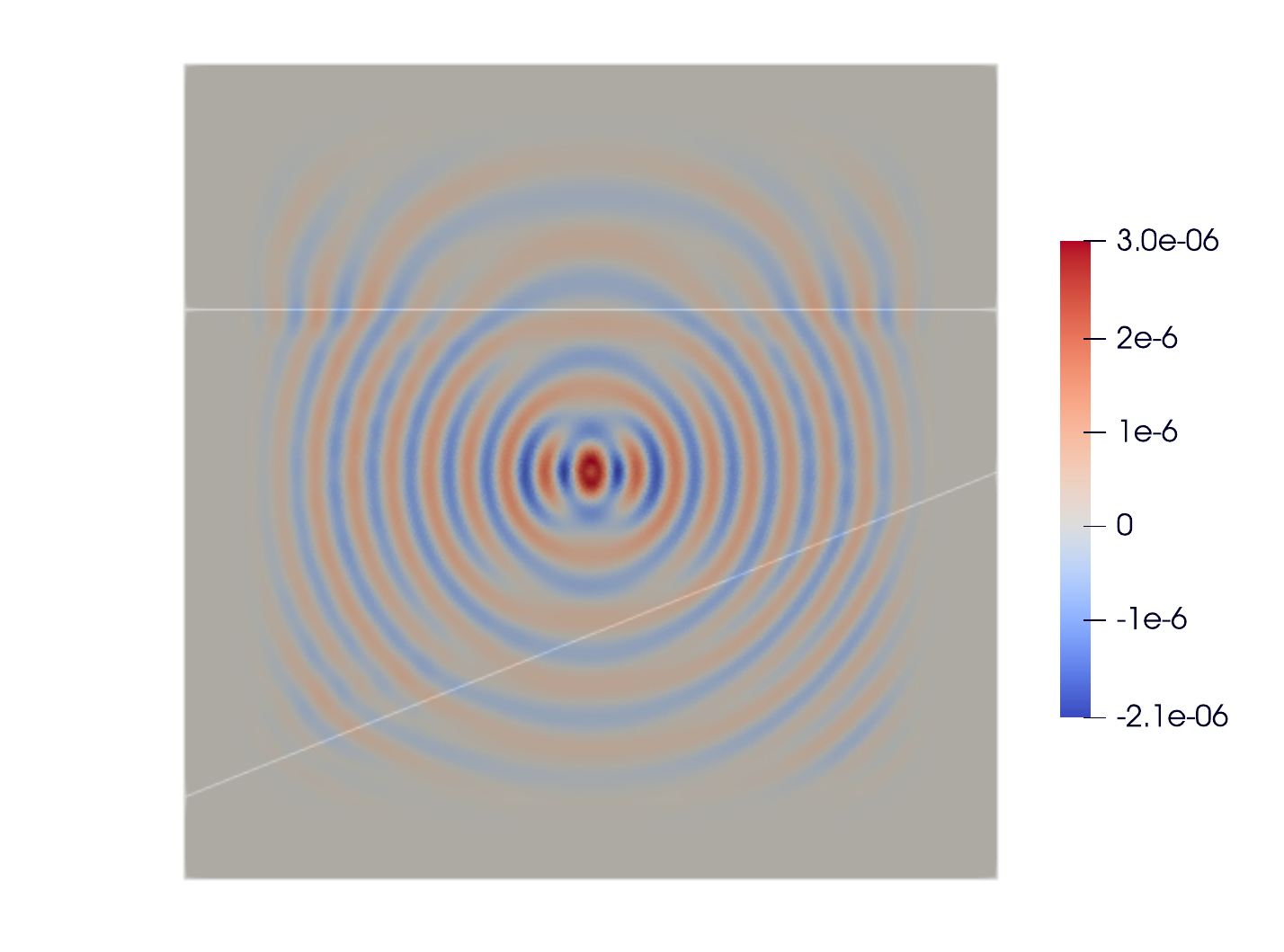}
			\caption{Imaginary part}
		\end{subfigure}
		
		\caption{Test case 2 - Numerical solution of $\textbf{v}_y$ at 30Hz with 130118 number of elements.}
		\label{fig: example 2 velocity y comp solution}
	\end{figure}
	The number of iterations for different stabilization parameters is reported in \Cref{fig: example 2 number of iterations}. Here, the number of iterations is again smallest for $L=\alpha^{2}/\beta$. We also observe that moving further away from this parameter results in significantly more iterations in the heterogeneous case. Again, for suboptimal stabilization parameters, the number of iterations increases with the frequency.

	\begin{figure}[H]
		\centering
		\includegraphics[width=0.5\linewidth]{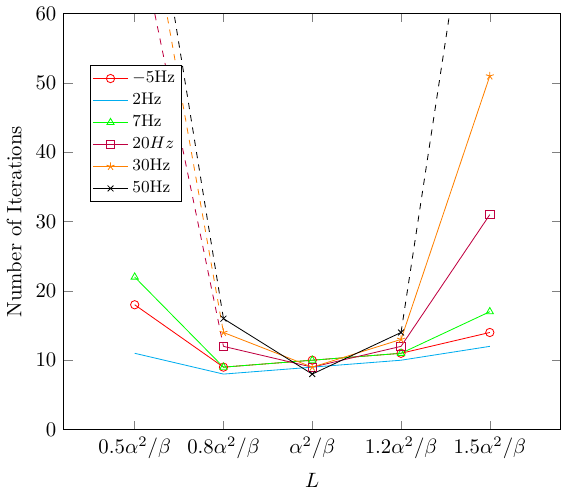}
		\caption{Test case 2 - Number of iterations for different stabilization parameters varying in space with the facies with 130118 number of elements. Dashed lines indicate convergence achieved after 100 iterations.}
		\label{fig: example 2 number of iterations}
		
	\end{figure}

	\subsection{Test Case 3 — 3D Anisotropic Heterogeneous Model (using PSM)}
	
	The previous examples in Sections~5.1 and~5.2 employed finite-element discretization in two-dimensional isotropic settings. To demonstrate that the proposed \emph{L-stabilized splitting scheme} is not tied to a specific spatial discretization and remains robust in more complex scenarios, we consider a fully three-dimensional model with \emph{vertical transverse isotropy (VTI)} and implement the solver using a \emph{pseudospectral method (PSM)}. This choice emphasizes the generality of the splitting approach and its applicability to large-scale problems on regular grids, where FFT-based techniques offer computational advantages. Pseudospectral methods have a long tradition in wave modeling and have been widely used in poroelastic simulations by Carcione and others~\cite{carcione1995, Fornberg1987, Boyd2001, Trefethen2000}.
		\begin{figure}
		\centering
		\includegraphics[width=0.95\linewidth]{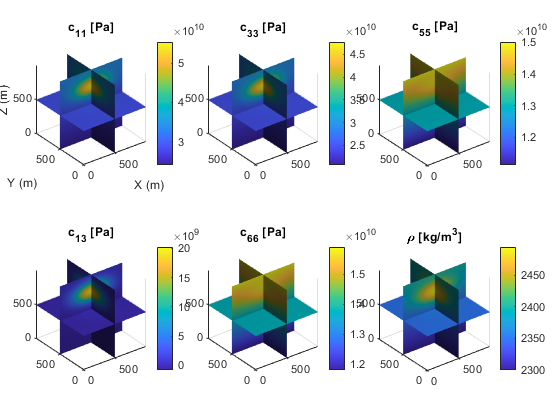}
		\caption{Test case 3 - Elastic stiffness components $(c_{11},c_{33},c_{55},c_{13},c_{66})$ and density $\rho$ of a factorized VTI medium with a Gaussian anomaly.}
	\end{figure}


	
	The computational domain measures $1000\,\mathrm{m} \times 1000\,\mathrm{m} \times 1000\,\mathrm{m}$ and is discretized on a uniform Cartesian grid with a grid size equal to 4.971 m.  The heterogeneous (but smooth) anisotropic poroelastic model can be decomposed into a factorized VTI model \cite{Xu2016DivingVTI}, with constant anisotropy parameters, depth-dependent isotropic parts of the tensorial coefficient fields and a Gaussian anomaly. The effective stress tensor was computed from the corresponding undrained stiffness field using standard tensorial poroelastic relations~\cite{Carcione2022}. In other words, the poroelastic constants were not selected independently of each other, and have physically realistic values. Figures 11 and 12 summarizes the poroelastic parameters used in this test case, including the elastic parameters and the the poro-elastic constants required to fully characterize a VTI poroelastic medium. As one can see in Figures 11 and 12, this 3D anisotropic poroelastic model is heterogeneous but smooth. Therefore, we expect that the pseudo-spectral method used for spatial discretization will be highly accurate. 
		
	\begin{figure}
		\centering
		\includegraphics[width = 0.95\linewidth]{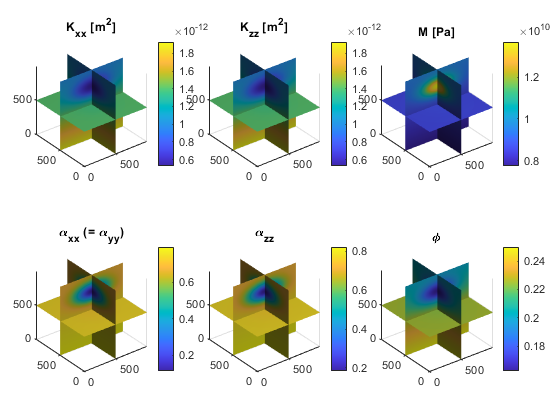}
		\caption{Test case 3 -Poroelastic parameters: $K_{xx}$, $K_{zz}$, $M$ (first row) and $\alpha _{xx}$, $\alpha _{zz}$, $\phi $ (second row). For the fluid viscosity, we used $\eta = 1 $ Cp. }
	\end{figure}
	
	As discussed in Appendix E, our implementation of the L-stabilized splitting scheme for this 3D example is matrix-free: spatial derivatives are computed via FFT/IFFT with wavenumber multipliers, while coefficient variations are applied pointwise in physical space. Therefore, our implementation has good memory management and is relatively efficient compared with the FEM implementation discussed earlier.  Krylov solvers (GMRES) handle block operator inversions efficiently without assembling large matrices~\cite{saad1986gmres}. To maintain accuracy for variable coefficients and mitigate Gibbs phenomena, we apply standard aliasing-control strategies such as the 2/3-rule and filtering~\cite{Fornberg1987,Boyd2001}. 
	Absorbing boundaries are implemented using complex coordinate stretching (PML) following Chew et al. \cite{chew1997} and Bérenger~\cite{berenger1994}. 
	More details about the PSM and our matrix-free implementation of the L-stabilized splitting scheme are provided in Appendix E. 
	
	Figure 13 show orthogonal slices of the computed wavefields $\Ree\{v_x\}$, $\Ree\{v_y\}$, $\Ree\{v_z\}$ and $\Ree\{\dot{p}\}$ at $f = 15\,\mathrm{Hz}$. The results confirm that the splitting scheme converges robustly in this fully anisotropic 3D setting, as illustrated by the iteration history in Figure~14. The total computation time for this 3D anisotropic heterogeneous test case with 8120601 grid blocks (performed using Matlab on a HP Precision work station with a Intel Xeon(R) w7-2475X CPU and a NVIDIA RTX A4000 GPU) was 12.5 min. For heterogeneous but smooth poroelastic models, the pseudospectral method provides excellent accuracy at moderate resolution. For sharp material interfaces, accurate representation of reflected and transmitted waves depends on grid resolution and aliasing control, as discussed in Appendix E. While FEM remains more flexible for complex geometries and unstructured meshes, this example demonstrates that the proposed splitting approach integrates seamlessly with FFT-based solvers, making it an attractive option for large-scale modeling on regular grids.

	\begin{figure}
		\centering
		\includegraphics[width = 0.95\linewidth]{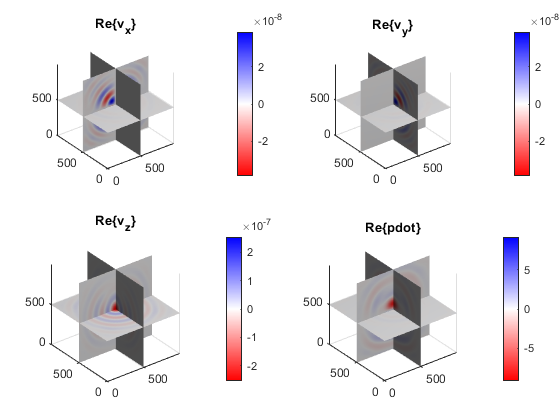}
		\caption{Test case 3 - Computed wavefields: orthogonal slices of $\Ree\{v_x\}$, $\Ree\{v_y\}$, $\Ree\{v_z\}$, and $\Ree\{\dot{p}\}$ at $f=15~\mathrm{Hz}$.}
	\end{figure}
	
	\begin{figure}[htbp]
		\begin{minipage}[c]{0.3\textwidth}
			\caption{Test case 3 - Outer and inner iteration convergence history for operator-$L$ stabilization. By outer iteration, we mean the iterations in the L-stabilized splitting scheme. By inner iterations, we mean the number of iterations required to solve the linear systems associated with the velocity and pressure rate equation at each outer iteration using the GMRES iterative method.}
			\label{fig:my_figure}
		\end{minipage}\hfill
		\begin{minipage}[c]{0.69\textwidth}
			\includegraphics[width=\textwidth]{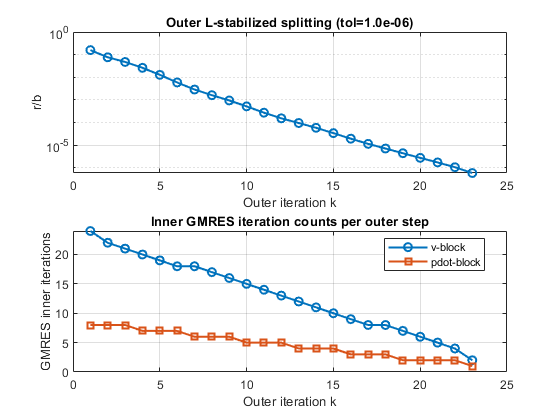}
		\end{minipage}
	\end{figure}

	\FloatBarrier

	\section*{6 Concluding Remarks}\label{sec:6}
	
	We formulated the dynamic Biot–Allard equations in the frequency domain and introduced
	a velocity–pressure-rate representation that restores operator structure under the energy inner
	product: conservative couplings become skew-adjoint in the quasi-static limit, dissipative cou-
	plings remain adjoint and positive-definite, enabling rigorous analysis of a stabilized splitting
	scheme with proven continuity and coercivity properties (up to the high-frequency scatter-
	ing limit). Numerical tests in homogeneous and heterogeneous media confirmed robustness
	and frequency-dependent attenuation. Most numerical results use FEM for compatibility with
	complex geological structures and operator analysis, but a 3D anisotropic example with PSM
	demonstrates discretization-agnostics and scalability. Future work may target slow P-waves and
	the inverse problem of determining the fluid mobility from seismic dispersion and attenuation.
	
	Finally, we remark that the present 2-field formulation, combined with an L-stabilized iterative splitting, provides a robust, efficient, and structurally well-founded framework for the problems considered here. It can be viewed as a natural extension to the dynamic regime of flow-oriented splitting strategies developed for quasi-static Biot systems, while still capturing the essential features of wave propagation, as demonstrated by the numerical results. As a complementary direction, ongoing work explores a reciprocity-aware 4-field formulation, closer in spirit to approaches commonly used in computational poroelastic wave modeling. By retaining the full block structure of the governing operators, this setting may offer additional flexibility for designing iterative splittings acting on physically meaningful sub-operators, while both perspectives remain consistent descriptions of the same underlying multiphysics system.
\FloatBarrier
	

	
	
	

	\section*{Acknowledgements} All authors acknowledge the support of the VISTA program, The Norwegian Academy of Science and Letters, and Equinor.


	\appendix

	\section{Time-Domain Formulation of Dynamic Biot Theory}
	\label{app:time_domain}
	
	For completeness, we summarize the time-domain formulation of dynamic Biot theory
	underlying the frequency-domain equations presented in Section~\ref{sec: 2}. The
	formulation is expressed in terms of the solid displacement $\mathbf{u}(\mathbf{x},t)$,
	the relative fluid--solid displacement
	\[
	\mathbf{w}(\mathbf{x},t) = \phi\big(\mathbf{U}(\mathbf{x},t)-\mathbf{u}(\mathbf{x},t)\big),
	\]
	and the pore pressure $p(\mathbf{x},t)$.
	
	We consider a fully saturated porous medium under small-strain conditions. Material
	properties may be anisotropic and spatially varying. Dissipative and dispersive
	effects are represented through temporal convolution operators.
	
	\medskip
	
	\noindent
	\textbf{Momentum balance for the solid phase.}
	\begin{equation}
		\nabla \cdot \boldsymbol{\sigma}
		- \rho\,\frac{\partial^2 \mathbf{u}}{\partial t^2}
		- \rho_f\,\frac{\partial^2 \mathbf{w}}{\partial t^2}
		+ \mathbf{f}
		= \mathbf{0},
		\label{eq:solid_momentum_time}
	\end{equation}
	where $\boldsymbol{\sigma}$ is the total stress tensor, $\mathbf{f}$ is an external
	body-force density, $\rho = (1-\phi)\rho_s + \phi\rho_f$ is the bulk density, and
	$\rho_f$ is the fluid density.
	
	\medskip
	
	\noindent
	\textbf{Fluid momentum equation}
	\begin{equation}
		\frac{\partial \mathbf{w}}{\partial t}
		+ \int_{-\infty}^{t} \mathbf{B}(t-\tau)
		\cdot
		\left[
		\nabla p(\tau)
		+ \rho_f \frac{\partial^2 \mathbf{u}}{\partial \tau^2}
		\right]
		\,\mathrm{d}\tau
		= \mathbf{0},
		\label{eq:fluid_momentum_time}
	\end{equation}
	where $\mathbf{B}(t)$ is a causal tensor-valued kernel related to the dynamic mobility.
	
	\medskip
	
	\noindent
	\textbf{Mass conservation.}
	\begin{equation}
		\frac{\partial}{\partial t}
		\left[
		\frac{p}{M}
		+ \boldsymbol{\alpha} : \boldsymbol{\varepsilon}(\mathbf{u})
		+ \nabla \cdot \mathbf{w}
		\right]
		- q
		= 0,
		\label{eq:mass_balance_time}
	\end{equation}
	where $M$ is the Biot modulus, $\boldsymbol{\alpha}$ is the Biot effective stress
	tensor, and $q$ is the rate of fluid-content change per unit bulk volume.
	
	\medskip
	
	\noindent
	\textbf{Constitutive relation.}
	\begin{equation}
		\boldsymbol{\sigma}(\mathbf{x},t)
		=
		\int_{-\infty}^{t}
		\mathbf{C}(t-\tau)
		:
		\boldsymbol{\varepsilon}(\mathbf{u})(\mathbf{x},\tau)
		\,\mathrm{d}\tau
		-
		\boldsymbol{\alpha}\,p(\mathbf{x},t),
		\label{eq:constitutive_time}
	\end{equation}
	where $\mathbf{C}(t)$ is a causal fourth-order tensor kernel and
	\[
	\boldsymbol{\varepsilon}(\mathbf{u})
	= \tfrac{1}{2}\big(\nabla \mathbf{u} + \nabla \mathbf{u}^\top\big)
	\]
	is the small-strain tensor.
	
	\medskip
	
	\noindent
	\textbf{Frequency-domain reduction.}
	Assuming time-harmonic fields of the form
	\[
	\mathbf{u}(\mathbf{x},t) = \mathbf{u}(\mathbf{x}) e^{i\omega t},
	\quad
	\mathbf{w}(\mathbf{x},t) = \mathbf{w}(\mathbf{x}) e^{i\omega t},
	\quad
	p(\mathbf{x},t) = p(\mathbf{x}) e^{i\omega t},
	\]
	and applying the Fourier transform in time with the convention $e^{i\omega t}$,
	temporal derivatives are replaced according to
	\[
	\frac{\partial}{\partial t} \;\mapsto\; i\omega,
	\qquad
	\frac{\partial^2}{\partial t^2} \;\mapsto\; -\omega^2.
	\]
	
	The convolution operators reduce to multiplication by frequency-dependent tensors,
	\[
	\mathbf{B}(\omega) = \int_{0}^{\infty} \mathbf{B}(t) e^{-i\omega t}\,\mathrm{d}t,
	\qquad
	\mathbf{C}(\omega) = \int_{0}^{\infty} \mathbf{C}(t) e^{-i\omega t}\,\mathrm{d}t.
	\]
	By a slight abuse of notation, the same symbol $\mathbf{B}$ is used for the
	causal time-domain kernel and its Fourier transform.
	Substitution into \eqref{eq:solid_momentum_time}--\eqref{eq:constitutive_time}
	then yields the frequency-domain system given by equations (2)--(4) in the main text.

	\setcounter{theorem}{0}

	\section{Useful identities}\label{sec: useful identities}
	\begin{lemma}[Binomial identity]
		\renewcommand{\thelemma}{1}
		\label{lem: stabilization identity}For any $x,y$ in a complex Hilbert space and any $L > 0$,
		\[
		\mathrm{Re}\,\langle L(x-y), x \rangle = \tfrac{L}{2}\left( \|x\|^2 + \|x-y\|^2 - \|y\|^2 \right).
		\]
	\end{lemma}
	
	
	\begin{lemma}[Difference estimate]\label{lem: difference estimate}
		If $A_{22} e_2^{k+1} + A_{21} e_1^{k+1} = 0$, then
		\[
		\|e_2^{k+1} - e_2^{k}\| \le \frac{\|A_{21}\|}{\alpha_{22}} \|e_1^{k+1} - e_1^{k}\|.
		\]
	\end{lemma}
	
	\begin{proof}
		Let $\delta e_2 = e_2^{k+1} - e_2^k$, $\delta e_1 = e_1^{k+1} - \bm{e}_{1}^{k}$. From the assumption:
		\[
		A_{22} e_2^{k+1} = -A_{21} e_1^{k+1}.
		\]
		Subtracting the same equation at step $k$:
		\[
		A_{22} \delta e_2 = -A_{21} \delta e_1.
		\]
		Taking norms:
		\[
		\|A_{22} \delta e_2\| = \|A_{21} \delta e_1\| \le \|A_{21}\| \cdot \|\delta e_1\|.
		\]
		Assuming $A_{22}$ is coercive with constant $\alpha_{22} > 0$:
		\[
		\|A_{22} x\| \ge \alpha_{22} \|x\| \quad \text{for all } x,
		\]
		we get:
		\[
		\alpha_{22} \|\delta e_2\| \le \|A_{21}\| \cdot \|\delta e_1\|,
		\]
		which gives the desired inequality.
	\end{proof}

	\begin{lemma}[Cauchy--Schwarz and Young's inequality]
		Let $x, y$ be elements of a Hilbert space. Then
		\[
		|\langle x, y \rangle| \le \|x\| \cdot \|y\|.
		\]
		Moreover, for all $a, b \ge 0$ and any $\tau > 0$, we have
		\[
		ab \le \frac{\tau}{2} a^2 + \frac{1}{2\tau} b^2.
		\]
	\end{lemma}
	

	\section{Properties of the diagonal block operators}\label{sec: appendix diagonal}

	\subsection{Coercivity of \texorpdfstring{$\calA_{11}$}{TEXT}}
	
	This section establishes coercivity estimates for the elastic and poroelastic operators used in the frequency-domain formulation of wave propagation, which are essential for the iterative splitting-scheme.
	
	\begin{lemma}[Coercivity of the elastic Helmholtz operator]
		Let $C_{\Omega}$ be the Poincaré constant of $\bm{H}_{0}^{1}$. 
		Fix $\omega \ge 0$ and $\rho \ge 0$, and define for $\mathbf{v} \in \bm{H}_{0}^{1}$
		\[
		A_{\mathrm{el}}\mathbf{v}
		= -\,\rho\,\omega^2\,\mathbf{v} 
		- \nabla \cdot \big(\mathbf{C}(\omega) : \nabla \mathbf{v}\big).
		\]
		Assume that the medium resists   
		\[
		\operatorname{Re}\big(\mathbf{C}(\omega)\mathbf{X}\big):\mathbf{X} \ge c_{\min} |\mathbf{X}|^2
		\quad \text{for all } \mathbf{X} \in \mathbb{C}^{d \times d} \text{ a.e. in } \Omega.
		\]
		Then, for all $\mathbf{v} \in \bm{H}_{0}^{1}$,
		\begin{equation}
			\operatorname{Re}\langle A_{\mathrm{el}}\mathbf{v}, \mathbf{v}\rangle 
			\ge \left( \frac{c_{\min}}{C_{\Omega}^2} - \rho \omega^2 \right) \|\mathbf{v}\|^2.
			\label{eq:elastic_coercivity_bound}
		\end{equation}
		In particular, $A_{\mathrm{el}}$ is coercive when
		$\frac{c_{\min}}{C_{\Omega}^2} > \rho \omega^2$.
	\end{lemma}
	
	\begin{proof}
		For $\mathbf{v} \in \bm{H}_{0}^{1}$, integration by parts (with vanishing boundary terms) yields
		\[
		\langle -\nabla\!\cdot(\mathbf{C}:\nabla\mathbf{v}),\,\mathbf{v}\rangle 
		= \langle \mathbf{C}:\nabla\mathbf{v},\,\nabla\mathbf{v}\rangle .
		\]
		Hence
		\[
		\operatorname{Re}\langle A_{\mathrm{el}}\mathbf{v}, \mathbf{v}\rangle 
		= -\,\rho\,\omega^2\,\|\mathbf{v}\|^2
		+ \operatorname{Re}\langle \mathbf{C}:\nabla\mathbf{v},\,\nabla\mathbf{v}\rangle .
		\]
		By ellipticity of $\operatorname{Re}\mathbf{C}$ and Poincaré’s inequality,
		\[
		\operatorname{Re}\langle \mathbf{C}:\nabla\mathbf{v},\,\nabla\mathbf{v}\rangle 
		\ge c_{\min}\,\|\nabla\mathbf{v}\|^2
		\ge \frac{c_{\min}}{C_{\Omega}^2}\,\|\mathbf{v}\|^2.
		\]
		Combining these bounds proves the claimed lower estimate. If the coefficient of $\|\mathbf{v}\|^2$ on the right-hand side is strictly positive, the operator is coercive on $\bm{L}^{2}$.
	\end{proof}
	
	\setcounter{theorem}{0}

	\begin{corollary}[Coercivity of the full operator $\calA_{11}$]
		Let $A_{\mathrm{el}}$ be the elastic Helmholtz operator defined previously, and assume the same conditions on the stiffness tensor $\mathbf{C}(\omega)$. Let $\mathbf{B}(\omega) \in \mathbb{C}^{d \times d}$ be a complex-valued, symmetric mobility tensor field defined almost everywhere in $\Omega$, representing fluid--solid coupling in a poroelastic medium. Define the full operator
		\begin{equation}
			\calA_{11}\mathbf{v} := A_{\mathrm{el}}\mathbf{v} + i\,\omega^3 \rho_f^2\,\mathbf{B}(\omega)\mathbf{v},
			\quad \text{for } \mathbf{v} \in \bm{H}_{0}^{1}.
			\label{Poroelastic_A11}
		\end{equation}
		Assume that the dissipative effect of $\operatorname{Im}\mathbf{B}(\omega)$ is bounded below in the sense that 
		\begin{equation}
			b_{\min}(\omega) := \inf_{\|\mathbf{q}\|=1} \langle \operatorname{Im}\mathbf{B}(\omega)\mathbf{q}, \mathbf{q}\rangle > 0.
			\label{Poroelastic_A11_Bmin}
		\end{equation}
		Then, for all $\mathbf{v} \in \bm{H}_{0}^{1}$,
		\[
		\operatorname{Re}\,\langle \calA_{11}\mathbf{v}, \mathbf{v}\rangle
		\ge \alpha_{11}(\omega)\,\|\mathbf{v}\|^2,
		\]
		where
		\begin{equation*}
			\alpha_{11}(\omega) := \frac{c_{\min}}{C_{\Omega}^2} - \rho \omega^2 + \omega^3 \rho_f^2 b_{\min}(\omega).
		\end{equation*}
		In particular, $\calA_{11}$ is coercive when $\alpha_{11}(\omega) > 0$.
	\end{corollary}
	
	\begin{proof}
		By the coercivity estimate for $A_{\mathrm{el}}$, we have
		\begin{equation*}
			\operatorname{Re}\,\langle A_{\mathrm{el}}\mathbf{v}, \mathbf{v} \rangle
			\ge \left( \frac{c_{\min}}{C_{\Omega}^2} - \rho \omega^2 \right)\|\mathbf{v}\|^2.
		\end{equation*}
		For the mobility term, we compute
		\begin{equation*}
			\operatorname{Re}\,\langle i\,\omega^3 \rho_f^2\,\mathbf{B}(\omega)\mathbf{v}, \mathbf{v} \rangle
			= \omega^3 \rho_f^2\,\langle \operatorname{Im}\mathbf{B}(\omega)\mathbf{v}, \mathbf{v} \rangle.
		\end{equation*}
		Using the lower bound on $\operatorname{Im}\mathbf{B}(\omega)$, we obtain
		\begin{equation*}
			\langle \operatorname{Im}\mathbf{B}(\omega)\mathbf{v}, \mathbf{v} \rangle
			\ge b_{\min}(\omega)\,\|\mathbf{v}\|^2.
		\end{equation*}
		Combining the estimates yields
		\begin{equation*}
			\operatorname{Re}\,\langle \calA_{11}\mathbf{v}, \mathbf{v} \rangle
			\ge \left( \frac{c_{\min}}{C_{\Omega}^2} - \rho \omega^2 + \omega^3 \rho_f^2 b_{\min}(\omega) \right)\|\mathbf{v}\|^2,
		\end{equation*}
		which proves the claim.
	\end{proof}
	
	\setcounter{theorem}{4}

	\begin{remark}[Physical interpretation of coercivity for the elastic Helmholtz operator]
		The coercivity estimate established in Lemma~3 reflects a fundamental physical balance between elastic restoring forces and inertial effects in the frequency-domain formulation of elastodynamics. Specifically, the operator \( A_{\mathrm{el}} \) is coercive whenever the contribution from the real part of the stiffness tensor dominates the inertial term \( -\rho \omega^2 \mathbf{v} \), as quantified in equation~\eqref{eq:elastic_coercivity_bound}.
		
		Physically, this means that the system stores sufficient elastic energy to resist oscillatory motion, ensuring that wave propagation remains stable and bounded. The condition \( \frac{c_{\min}}{C_{\Omega}^2} > \rho \omega^2 \) corresponds to a low-frequency regime in which elastic forces are strong enough to prevent resonance or instability. Coercivity in this context guarantees that the variational formulation is well-posed and that the system does not exhibit unphysical amplification of energy.
	\end{remark}

	\begin{remark}[Physical interpretation of coercivity for the full poroelastic operator]
		The coercivity of the operator \( \calA_{11} \), defined in \eqref{Poroelastic_A11}, reflects a fundamental physical property of the poroelastic system: its ability to dissipate energy through internal friction. Physically, coercivity ensures that wave propagation in the medium is stable and that energy does not grow unboundedly. This corresponds to damping mechanisms that arise from fluid–solid interactions, which convert mechanical energy into heat or other forms of loss.
		
		In the mathematical formulation, this dissipative behavior is captured by the imaginary part of the mobility tensor \( \mathbf{B}(\omega) \). The term \( i\,\omega^3 \rho_f^2\,\mathbf{B}(\omega)\mathbf{v} \) models dynamic coupling between the fluid and solid phases, and its contribution to the real part of the energy inner product is strictly positive under the assumption that \( \operatorname{Im}\mathbf{B}(\omega) \) is uniformly positive semidefinite.
		
		This assumption, stated in equation~\eqref{Poroelastic_A11_Bmin}, ensures that the mobility term enhances the coercivity of \( \calA_{11} \). It reflects the physically reasonable condition that energy dissipation occurs in all directions of motion. Together, these properties guarantee that the variational formulation associated with \( \calA_{11} \) is stable and well-posed.
	\end{remark}
	
	\begin{lemma}[Coercivity of $\mathcal{A}_{22}$]
		
		For $\omega > 0$ and $\dot p \in H^1_0$, we define
		\[
		\mathcal{A}_{22}\dot p := \frac{1}{M}\dot p - \frac{1}{i\omega}\nabla\cdot(\mathbf{B}(\omega)\nabla\dot p), \quad
		\mathbf{B}_\mathrm{H}(\omega) := \frac{\mathbf{B}(\omega) - \mathbf{B}(\omega)^\dagger}{2i}.
		\]
		Assume that a.e. in $\Omega$, it holds that $\operatorname{Re}(1/M) \ge \beta_* \ge 0$ and $\mathbf{B}_\mathrm{H}(\omega) \succeq \gamma_*\,\mathbf{I}$ for some $\gamma_* \ge 0$. 
		Then, it holds that
		\[
		\operatorname{Re}\,\langle \mathcal{A}_{22}\dot p,\,\dot p\rangle
		\ge \alpha_{22}(\om)\|\dot p\|^2, \quad \alpha_{22}(\om):=\left(\beta_*+\frac{\gamma_*}{\omega}C_{\Omega}^{-2}\right), 
		\]
		with $C_{\Omega}$ being the Poincar\'e constant.
	\end{lemma}
	
	\begin{proof}
		For $\dot p\in H_{0}^{1}$, after integration by parts gives
		\begin{equation*}
			\langle \mathcal{A}_{22}\dot p,\,\dot p\rangle
			= \int_\Omega \frac{1}{M}|\dot p|^2\,d\bm{x}
			+ \frac{1}{i\omega}\int_\Omega \mathbf{B}(\omega) \cdot \nabla\dot p \cdot \overline{\nabla\dot p}\,d\bm{x}. 
		\end{equation*}
		We take the real part of the inner product
		\begin{equation*}
			\operatorname{Re}\,\langle \mathcal{A}_{22}\dot p,\,\dot p\rangle
			= \int_\Omega \operatorname{Re}\left(\frac{1}{M}\right)|\dot p|^2\,d\bm{x}
			+ \frac{1}{\omega}\int_\Omega \mathbf{B}_\mathrm{H}(\omega) \cdot \nabla\dot p \cdot \overline{\nabla\dot p}\,d\bm{x}. 
		\end{equation*}
		By the assumption on $\mathbf{B}_\mathrm{H}(\omega)$ we have 
		\begin{equation*}
			\int_\Omega \mathbf{B}_\mathrm{H}(\omega) \cdot \nabla\dot p \cdot \overline{\nabla\dot p}\,d\bm{x}
			\ge \gamma_* \|\nabla\dot p\|^2. 
		\end{equation*}
		It follows that
		\begin{equation*}
			\operatorname{Re}\,\langle \mathcal{A}_{22}\dot p,\,\dot p\rangle
			\ge \beta_*\,\|\dot p\|^2
			+ \frac{\gamma_*}{\omega}\,\|\nabla\dot p\|^2. 
		\end{equation*}
		Next by using the \'e inequality with constant $C_{\Omega}$ we get
		\begin{equation*}
			\|\nabla\dot p\|^2 \ge C_{\Omega}^{-2}\|\dot p\|^2, 
		\end{equation*}
		Therefore we have
		\begin{equation*}
			\operatorname{Re}\,\langle \mathcal{A}_{22}\dot p,\,\dot p\rangle
			\ge \left(\beta_*+\frac{\gamma_*}{\omega}C_{\Omega}^{-2}\right)\|\dot p\|^2. 
		\end{equation*}
	\end{proof}
	
	\begin{remark}[Physical interpretation of coercivity]
		$\mathcal{A}_{22}$ is coercive for all $\omega > 0$ if the assumptions in Lemma 8 hold. The coercivity constant
		\[
		\alpha_{22} = \beta_* + \frac{\gamma_*}{\omega} C_{\Omega}^{-2}
		\]
		shows that storage dominates at high frequencies (via $\beta_*$), while dissipation dominates at low frequencies (via $\gamma_*$). The Hermitian tensor $\mathbf{B}_\mathrm{H}(\omega)$ is real and positive definite, representing the imaginary (dissipative) part of the mobility tensor. This ensures monotonic energy dissipation in the sense of the sesquilinear form. Consequently, the system either stores or dissipates energy, guaranteeing bounded pressure fluctuations and dynamic stability.
	\end{remark}

	\section{Properties of the off-diagonal block operators}\label{sec: appendix A}
	
	\subsection{Quasi skew-adjoint relationship between coupling operators}
	
	We use the term \emph{quasi skew-adjoint} to indicate that the coupling operators become skew-adjoint in the quasi-static limit, while this structure is lost at finite frequencies due to dissipative effects.
	\color{black}

	\begin{lemma}[Skew-adjointness of quasi-static coupling operators]\label{lem: adjoint skew}
		Let $\mathcal{A}_{12}^{\alpha}$ and $\mathcal{A}_{21}^{\alpha}$ denote the quasi-static parts of the coupling operators in the velocity--pressure rate formulation of Biot's equations, defined by
		\begin{align*}
			\mathcal{A}_{12}^{\alpha}\dot{p} = \nabla \cdot (\boldsymbol{\alpha}\,\dot{p}),\quad
			\mathcal{A}_{21}^{\alpha}\mathbf{v} = \boldsymbol{\alpha} : \boldsymbol{\varepsilon}(\mathbf{v}),
		\end{align*}
		where $\boldsymbol{\alpha}$ is a symmetric second-order tensor and $\boldsymbol{\varepsilon}(\mathbf{v}) = \tfrac{1}{2}(\nabla \mathbf{v} + \nabla \mathbf{v}^\top)$ is the strain rate tensor. Assume homogeneous boundary conditions such that all boundary integrals vanish. Then $\mathcal{A}_{12}^{\alpha}$ and $\mathcal{A}_{21}^{\alpha}$ satisfy the skew-adjoint relationship
		\[
		\langle \mathcal{A}_{12}^{\alpha}\dot{p}, \mathbf{v} \rangle = - \langle \dot{p}, \mathcal{A}_{21}^{\alpha}\mathbf{v} \rangle,
		\]
		with respect to the $L^2$ inner product.
	\end{lemma}
	
	\begin{proof}
		Consider the $L^2$ inner product of $\mathcal{A}_{12}^{\alpha}\dot{p}$ with $\mathbf{v}$:
		\[
		\langle \mathcal{A}_{12}^{\alpha}\dot{p}, \mathbf{v} \rangle
		= \int_{\Omega} \big( \nabla \cdot (\boldsymbol{\alpha}\,\dot{p}) \big) \cdot \mathbf{v}\, d\bm{x}.
		\]
		Applying the divergence theorem and assuming that the boundary term vanishes under the given conditions, we obtain
		\[
		\int_{\Omega} \big( \nabla \cdot (\boldsymbol{\alpha}\,\dot{p}) \big) \cdot \mathbf{v}\, d\bm{x}
		= - \int_{\Omega} (\boldsymbol{\alpha}\,\dot{p}) : \nabla \mathbf{v}\, d\bm{x}.
		\]
		Since $\boldsymbol{\alpha}$ is symmetric, the double contraction with $\nabla \mathbf{v}$ reduces to the contraction with the symmetric part:
		\[
		(\boldsymbol{\alpha}\,\dot{p}) : \nabla \mathbf{v} = \dot{p}\, (\boldsymbol{\alpha} : \boldsymbol{\varepsilon}(\mathbf{v})).
		\]
		Thus,
		\[
		\langle \mathcal{A}_{12}^{\alpha}\dot{p}, \mathbf{v} \rangle
		= - \int_{\Omega} \dot{p}\, (\boldsymbol{\alpha} : \boldsymbol{\varepsilon}(\mathbf{v}))\, d\bm{x}
		= - \langle \dot{p}, \mathcal{A}_{21}^{\alpha}\mathbf{v} \rangle.
		\]
		This proves the skew-adjointness of $\mathcal{A}_{12}^{\alpha}$ and $\mathcal{A}_{21}^{\alpha}$ under the stated assumptions.
	\end{proof}

	\begin{lemma}[Adjointness of the dynamic $B$--coupling terms]
		Let $\bmv\in \bm{H}_{0}^{1}$ and $\dot{p}\in H^{1}_{0}$. Consider the frequency-dependent dynamic coupling operators
		\[
		(A_{12}^B \dot p)_j := -\,i\omega \rho_f\, B_{j k}\,\partial_k \dot p,
		\qquad
		A_{21}^B \mathbf{v} := -\,i\omega \rho_f\, \partial_k\!\big(B_{k j}\,v_j\big),
		\]
		with a sufficiently regular tensor field $B(\omega,\mathbf{x})\in\mathbb{C}^{d\times d}$.
		Then the off--diagonal operators satisfy
		\[
		A_{12}^B \;=\; (A_{21}^B)^{*}
		\quad\Longleftrightarrow\quad
		B_{k j} = \overline{B_{j k}},
		\]
		with respect to the complex $L^2$ inner products
		In particular, the condition holds when $\bm{B}$ is hermitian.
	\end{lemma}
	
	\begin{proof}
		\emph{Step 1 (integration by parts).}
		For $\bmw\in \bm{H}_{0}^{1}$ and $\dot{p}\in H^{1}_{0}$, we have
		\[
		\begin{aligned}
			\langle \calA_{12}^B \dot p, \mathbf{w}\rangle
			&= \int_\Omega \big(-i\omega\rho_f\, B_{j k}\,\partial_k \dot p\big)\,\overline{w_j}\,d\bm{x}
			= -\,i\omega\rho_f \int_\Omega B_{j k}\,(\partial_k \dot p)\,\overline{w_j}\,d\bm{x}
			\\
			&= i\omega\rho_f \int_\Omega \dot p\,\partial_k\big(B_{j k}\,\overline{w_j}\big)\,d\bm{x}
			\qquad\text{(by IBP; boundary term vanishes)}
			\\
			&= i\omega\rho_f \int_\Omega \dot p\,\overline{\partial_k\big(\overline{B_{j k}}\,w_j\big)}\,d\bm{x}
			= \omega\rho_f \int_\Omega \dot p\,\overline{(-i)\partial_k\big(\overline{B_{j k}}\,w_j\big)}\,d\bm{x}.
		\end{aligned}
		\]
		
		\emph{Step 2 (adjoint identification).}
		By the definition of the adjoint,
		\[
		\langle A_{12}^B \dot p, \mathbf{w}\rangle
		= \langle \dot p,(A_{12}^B)^{*}\mathbf{w}\rangle
		\quad\Rightarrow\quad
		(A_{12}^B)^{*}\mathbf{w} = -i\omega\rho_f\,\partial_k\big(\overline{B_{j k}}\,w_j\big).
		\]
		On the other hand,
		\[
		A_{21}^B \mathbf{w} = -\,i\omega\rho_f\,\partial_k\big(B_{k j}\,w_j\big).
		\]
		Thus, $A_{12}^B = (A_{21}^B)^{*}$ iff
		\[
		\partial_k\big(B_{k j}\,w_j\big) = \partial_k\big(\overline{B_{j k}}\,w_j\big)\quad\forall\,\mathbf{w}
		\;\;\Longleftrightarrow\;\;
		B_{k j} = \overline{B_{j k}},
		\]
		which is the Hermiticity of $\bm{B}$.
	\end{proof}

	\setcounter{theorem}{1}
	
	\begin{corollary}[The lack of a simple adjoint relationship between the full off-diagonal coupling operators]
		Let
		\[
		\mathcal{A}_{12} = \mathcal{A}_{12}^{\alpha} + \mathcal{A}_{12}^{\mathrm{B}}, 
		\qquad
		\mathcal{A}_{21} = \mathcal{A}_{21}^{\alpha} + \mathcal{A}_{21}^{\mathrm{B}},
		\]
		where the components are defined as in Lemma~6 and Lemma~7. 
		Assume that $\alpha$ is constant and $\mathbf{B}$ is Hermitian. 
		Under homogeneous Dirichlet or Neumann boundary conditions (or decay at infinity), the following relation holds with respect to the complex $L^2$ inner product:
		\[
		\mathcal{A}_{12} = -(\mathcal{A}_{21}^{\alpha})^* + (\mathcal{A}_{21}^{\mathrm{B}})^*.
		\]
	\end{corollary}
	
	\begin{proof}
		By Lemma~6,
		\[
		\mathcal{A}_{12}^{\alpha} = -(\mathcal{A}_{21}^{\alpha})^*,
		\]
		provided $\alpha$ is constant. 
		By Lemma~7 (for Hermitian $\mathbf{B}$),
		\[
		\mathcal{A}_{12}^{\mathrm{B}} = (\mathcal{A}_{21}^{\mathrm{B}})^*.
		\]
		Adding these relations gives
		\[
		\mathcal{A}_{12} = \mathcal{A}_{12}^{\alpha} + \mathcal{A}_{12}^{\mathrm{B}} 
		= -(\mathcal{A}_{21}^{\alpha})^* + (\mathcal{A}_{21}^{\mathrm{B}})^*,
		\]
		which proves the stated relation.
	\end{proof}
	
	\begin{remark}
		The full off-diagonal coupling in dynamic Biot theory combines two fundamentally different structures:
		\emph{Effective stress coupling} ($\alpha$-terms) is skew-adjoint, representing conservative energy transfer between solid and fluid.
		\emph{Darcy-flow coupling} ($B$-terms) is adjoint and positive-definite when $B$ is Hermitian, representing dissipative effects.
		Adjointness of the $B$-terms reflects their role in energy dissipation through viscous drag, which contributes a positive real part to the energy inner product and ensures damping and stability.
		
		This mixed structure explains why the operator matrix cannot be classified as self-adjoint or skew-adjoint.
		It also motivates operator-splitting strategies such as the Strang-splitting approach used by Carcione et al., where the evolution operator is decomposed into a conservative (skew-adjoint) part and a dissipative (positive-definite) part.
		Such splitting preserves stability and energy properties by treating the two components according to their physical nature.
	\end{remark}
	
	\setcounter{theorem}{7}
	
	\subsection{Continuity of coupling operators}
	
	\begin{lemma}[Continuity (equivalently, boundedness) of the coupling operator $\mathcal{A}_{12}$]\label{lem:A12-continuity}
		Assume
		\[
		\boldsymbol{\alpha}\in L^\infty(\Omega;\mathbb{C}^{d\times d})
		\quad\text{with}\quad
		\nabla\!\cdot\boldsymbol{\alpha}\in L^\infty(\Omega;\mathbb{C}^d),
		\qquad
		\mathbf{B}(\omega)\in L^\infty(\Omega;\mathbb{C}^{d\times d}).
		\]
		Define
		\[
		\mathcal{A}_{12}\dot p
		:= \nabla\!\cdot(\boldsymbol{\alpha}\,\dot p) - i\,\omega\,\rho_f\,\mathbf{B}(\omega)\,\nabla \dot p .
		\]
		Then $\mathcal{A}_{12}: H^1\to \bm{L}^{2}$ is a continuous linear operator. In particular, there exists a continuity constant
		\[
		L_{12}(\omega)
		:= \|\nabla\!\cdot\boldsymbol{\alpha}\|_{L^\infty}
		+ \|\boldsymbol{\alpha}\|_{L^\infty}
		+ |\omega|\,\rho_f\,\|\mathbf{B}(\omega)\|_{L^\infty}
		\]
		such that, for all $\dot p\in H^1$,
		\[
		\|\mathcal{A}_{12}\dot p\|
		\;\le\;
		L_{12}(\omega)\,\|\dot p\|_{H^1}.
		\]
		Equivalently, $\mathcal{A}_{12}$ is bounded with $\|\mathcal{A}_{12}\|\le L_{12}(\omega)$.
	\end{lemma}
	
	\begin{proof}
		For $\dot p\in H^1$, the product rule for weak derivatives yields
		\[
		\nabla\!\cdot(\boldsymbol{\alpha}\,\dot p)
		= (\nabla\!\cdot\boldsymbol{\alpha})\,\dot p + \boldsymbol{\alpha}\,\nabla \dot p
		\quad\text{in } \bm{\mathscr{D}}'.
		\]
		With the stated assumptions,
		\[
		\|(\nabla\!\cdot\boldsymbol{\alpha})\,\dot p\|
		\le \|\nabla\!\cdot\boldsymbol{\alpha}\|_{L^\infty}\,\|\dot p\|,
		\qquad
		\|\boldsymbol{\alpha}\,\nabla \dot p\|
		\le \|\boldsymbol{\alpha}\|_{L^\infty}\,\|\nabla \dot p\|,
		\]
		and
		\[
		\|\mathbf{B}(\omega)\,\nabla \dot p\|
		\le \|\mathbf{B}(\omega)\|_{L^\infty}\,\|\nabla \dot p\|.
		\]
		Combining these and using $\|\dot p\|\le \|\dot p\|_{H^1}$ and
		$\|\nabla \dot p\|\le \|\dot p\|_{H^1}$ gives
		\[
		\|\mathcal{A}_{12}\dot p\|
		\le \Big(\|\nabla\!\cdot\boldsymbol{\alpha}\|_{L^\infty}
		+ \|\boldsymbol{\alpha}\|_{L^\infty}
		+ |\omega|\,\rho_f\,\|\mathbf{B}(\omega)\|_{L^\infty}\Big)\,\|\dot p\|_{H^1}.
		\]
		Since the above inequality shows that $\mathcal{A}_{12}$ is bounded, the Bounded Linear Operator Theorem implies that $\mathcal{A}_{12}$ is also continuous.
	\end{proof}

	\begin{lemma}[Continuity (equivalently, boundedness) of the coupling operator $\mathcal{A}_{21}$]\label{lem:A21-continuity}
		Assume
		\[
		\boldsymbol{\alpha}\in L^\infty(\Omega;\mathbb{C}^{d\times d}),
		\quad
		\nabla\!\cdot\boldsymbol{\alpha}\in L^\infty(\Omega;\mathbb{C}^d),
		\qquad
		\mathbf{B}(\omega)\in L^\infty(\Omega;\mathbb{C}^{d\times d}),
		\quad
		\nabla\!\cdot\mathbf{B}(\omega)\in L^\infty(\Omega;\mathbb{C}^d),
		\]
		for each fixed frequency $\omega\in\mathbb{R}$. Define
		\[
		\mathcal{A}_{21}\mathbf{v}
		:= \nabla\!\cdot(\boldsymbol{\alpha}\,\mathbf{v}) - i\,\omega\,\rho_f\,\nabla\!\cdot(\mathbf{B}(\omega)\,\mathbf{v}),
		\qquad \mathbf{v}\in \bm{H}^{1}.
		\]
		Then $\mathcal{A}_{21}:\bm{H}^{1}\to L^2$ is a continuous linear operator. In particular, there exists a continuity constant
		\[
		L_{21}(\omega)
		:= \|\nabla\!\cdot\boldsymbol{\alpha}\|_{L^\infty}
		+ \|\boldsymbol{\alpha}\|_{L^\infty}
		+ |\omega|\,\rho_f\big(\|\nabla\!\cdot\mathbf{B}(\omega)\|_{L^\infty}
		+ \|\mathbf{B}(\omega)\|_{L^\infty}\big),
		\]
		such that, for all $\mathbf{v}\in\bm{H}^{1}$,
		\[
		\|\mathcal{A}_{21}\mathbf{v}\|
		\;\le\;
		L_{21}(\omega)\,\|\mathbf{v}\|_{H^1}.
		\]
		Equivalently, $\mathcal{A}_{21}$ is bounded with $\|\mathcal{A}_{21}\|\le L_{21}(\omega)$.
	\end{lemma}
	
	\begin{proof}
		For $\mathbf{v}\in\bm{H}^{1}$, the product rule for weak derivatives yields
		\[
		\nabla\!\cdot(\boldsymbol{\alpha}\,\mathbf{v})
		= (\nabla\!\cdot\boldsymbol{\alpha})\cdot\mathbf{v} + \boldsymbol{\alpha}:\nabla\mathbf{v},
		\qquad
		\nabla\!\cdot(\mathbf{B}(\omega)\,\mathbf{v})
		= (\nabla\!\cdot\mathbf{B}(\omega))\cdot\mathbf{v} + \mathbf{B}(\omega):\nabla\mathbf{v}
		\quad\text{in } \mathscr{D}'.
		\]
		With the stated assumptions,
		\[
		\|(\nabla\!\cdot\boldsymbol{\alpha})\cdot\mathbf{v}\|
		\le \|\nabla\!\cdot\boldsymbol{\alpha}\|_{L^\infty}\,\|\mathbf{v}\|,
		\qquad
		\|\boldsymbol{\alpha}:\nabla\mathbf{v}\|
		\le \|\boldsymbol{\alpha}\|_{L^\infty}\,\|\nabla\mathbf{v}\|,
		\]
		and similarly for $\mathbf{B}(\omega)$:
		\[
		\|(\nabla\!\cdot\mathbf{B}(\omega))\cdot\mathbf{v}\|
		\le \|\nabla\!\cdot\mathbf{B}(\omega)\|_{L^\infty}\,\|\mathbf{v}\|,
		\qquad
		\|\mathbf{B}(\omega):\nabla\mathbf{v}\|
		\le \|\mathbf{B}(\omega)\|_{L^\infty}\,\|\nabla\mathbf{v}\|.
		\]
		Combining these and using $\|\mathbf{v}\|\le \|\mathbf{v}\|_{\bm{H}^1}$ and
		$\|\nabla\mathbf{v}\|\le \|\mathbf{v}\|_{\bm{H}^1}$ gives
		\[
		\|\mathcal{A}_{21}\mathbf{v}\|
		\le \Big(\|\nabla\!\cdot\boldsymbol{\alpha}\|_{L^\infty}
		+ \|\boldsymbol{\alpha}\|_{L^\infty}
		+ |\omega|\,\rho_f(\|\nabla\!\cdot\mathbf{B}(\omega)\|_{L^\infty}
		+ \|\mathbf{B}(\omega)\|_{L^\infty})\Big)\,\|\mathbf{v}\|_{\bm{H}^{1}}.
		\]
		Since the above inequality shows that $\mathcal{A}_{21}$ is bounded, the Bounded Linear Operator Theorem implies that $\mathcal{A}_{21}$ is also continuous.
	\end{proof}

	\section{Alternative Discretization and Implementation}\label{sec: psm}
	
	\subsection{The pseudo-spectral method} 
	
	This paper focuses primarily on the finite element method (FEM) for spatial discretization, as it enables a rigorous operator-theoretic analysis of the coupled multi-physics problem. However, the proposed $L$-stabilized splitting scheme for the two-field velocity–pressure-rate formulation of the Biot–Allard equations is \emph{discretization-agnostic}. To illustrate this, we also consider a pseudo-spectral method (PSM).
	Spectral methods arise from a Galerkin formulation with global basis functions, whereas pseudo-spectral methods provide an efficient approximation based on collocation and FFT-based differentiation \cite{Gottlieb1977,Canuto1988,Fornberg1987,Kosloff1982}. 
	While care is required in the presence of spatial heterogeneities—such as internal interfaces, where filtering or staggered-grid techniques may be employed—pseudo-spectral methods remain highly effective for wave propagation due to their low numerical dispersion and computational efficiency. 
	
	The pseudo-spectral method relies on the observation that spatial derivatives can be evaluated efficiently using the Fourier transform; that is,
	\begin{equation}
		\partial_j u(x) = \mathcal{F}^{-1}\big( i k_j \,\mathcal{F}[u] \big),
		\label{PS_derivative}
	\end{equation} 
	where $\mathbf{k} = (k_1,k_2,k_3)$ denotes the wave vector and $u$ is an arbitrary scalar field or tensor component. 
	In practice, derivatives are computed in three steps: (i) transform the field to Fourier space, (ii) multiply by $i k_j$, and (iii) transform back to physical space. All material parameters and coefficient fields are evaluated pointwise in physical space. This results in an algorithm with $O(N\log N)$ complexity, where $N$ is the number of grid points.

	\subsection{Block operators in abbreviated subscript notation}
	
	To make the tensor structure explicit and compatible with FFT-based evaluation, we use the abbreviated subscript (Auld–Voigt) notation described in \cite{Auld1973, Carcione2022}, representing symmetric second-rank tensors as 6-vectors and fourth-order tensors with corresponding second-order index structure. Small indices $i,j = 1,2,3$ denote spatial components, capital indices $I,J,K = 1,\dots,6$ denote Voigt components, and Einstein summation convention is used.
	
	In this notation, the strain rate-displacement operator and its adjoint are
	\[
	\varepsilon_I = \nabla_{I j} v_j,
	\qquad
	(\nabla \cdot \boldsymbol{\sigma})_i = \nabla_{iJ} \sigma_J.
	\]
	
	The differential operators $\nabla_{Ij}$ and $\nabla_{iJ}$ are represented by constant-coefficient matrices. In three dimensions, using the standard Voigt ordering $I = (11,22,33,23,13,12)$, we define
	\[
	\nabla_{Ij} =
	\begin{pmatrix}
		\partial_1 & 0 & 0 \\
		0 & \partial_2 & 0 \\
		0 & 0 & \partial_3 \\
		0 & \partial_3 & \partial_2 \\
		\partial_3 & 0 & \partial_1 \\
		\partial_2 & \partial_1 & 0
	\end{pmatrix},
	\qquad
	\nabla_{iJ} = (\nabla_{J i})^{\top}.
	\]
	These operators map vector fields to Voigt-form strain and Voigt-form stress rates to their divergence (see, e.g., Auld~\cite{Auld1973} and Carcione~\cite{Carcione2022}).
	
	The block operators in equation~(\ref{eq:block-operators-vpr}) can then be written in physical space as
	\[
	\begin{aligned}
		(A_{11} v)_i 
		&= -\omega^2 \rho^{\mathrm{eff}}_{ij} v_j
		- \nabla_{iJ} \big( C_{JK} \, \nabla_{Kj} v_j \big), \\
		(A_{12} \dot p)_i 
		&= \nabla_{iJ}(\alpha_J \dot p)
		- i\omega \rho_f \, B_{ij} \, \partial_j \dot p, \\
		(A_{21} v) 
		&= \alpha_J \, \nabla_{Jj} v_j
		- i\omega \rho_f \, \partial_i \big(B_{ij} v_j\big), \\
		(A_{22} \dot p) 
		&= \frac{\dot p}{M}
		- \frac{1}{i\omega} \, \partial_i \big( B_{ij} \, \partial_j \dot p \big).
	\end{aligned}
	\]
	This formulation is consistent with the operator framework of Section~3 and makes the gradient–divergence structure explicit for FFT-based evaluation.

	\subsection{Pseudo-spectral evaluation of the operators}
	
	In the pseudo-spectral implementation, all spatial derivatives are evaluated through Fourier transforms, while fields are stored in physical space, as shown in equation \eqref{PS_derivative}. More generally, the Voigt-form operators $\nabla_{Ij}$ and $\nabla_{iJ}$ are evaluated in Fourier space by replacing each occurrence of $\partial_\ell$ with multiplication by $i k_\ell$. This defines the matrices $k_{Ij}$ and $k_{iJ}$ directly from $\nabla_{Ij}$ and $\nabla_{iJ}$ through the substitutions
	\[
	\nabla _{iJ} \;\mapsto\; ik_{iJ}, ~~~~\nabla _{Kl} \;\mapsto\; ik_{Kl},
	\]
	where $k_{iJ}$ and $k_{Kl}$ has the same structures as $\nabla _{iJ}$ and $\nabla _{Kl}$, respectively, but with $\partial _{i}$ replaced by the wave vector component $k_{i}$, $i= 1,3,3$. A similar recipe is used for plane wave analysis of elastic waves in anisotropic solids by Auld \cite{Auld1973} and Carcione also details this in his book \cite{Carcione2022}. 
	
	Applying this rule componentwise yields the following pseudo-spectral approximations of the block operators:
	\[
	\begin{aligned}
		(\mathcal{A}_{11} v)_i 
		&\approx -\omega^2 \rho^{\mathrm{eff}}_{ij} v_j 
		- \mathcal{F}^{-1}\Bigg( i k_{iJ} \,\mathcal{F}\Big[
		C_{JK} \,
		\mathcal{F}^{-1}\big( i k_{Kj} \,\mathcal{F}[v_j] \big)
		\Big]\Bigg), \\
		(\mathcal{A}_{12} \dot p)_i 
		&\approx \mathcal{F}^{-1}\Big( i k_{iJ} \,\mathcal{F}[ \alpha_J \dot p ] \Big)
		- i\omega \rho_f \, B_{ij} \,
		\mathcal{F}^{-1}\big( i k_j \,\mathcal{F}[\dot p] \big), \\
		(\mathcal{A}_{21} v) 
		&\approx \alpha_J \,
		\mathcal{F}^{-1}\big( i k_{Jj} \,\mathcal{F}[v_j] \big)
		- i\omega \rho_f \,
		\mathcal{F}^{-1}\Big( i k_i \,\mathcal{F}[ B_{ij} v_j ] \Big), \\
		(\mathcal{A}_{22} \dot p) 
		&\approx \frac{\dot p}{M}
		- \frac{1}{i\omega} \,
		\mathcal{F}^{-1}\Bigg( i k_i \,\mathcal{F}\Big[
		B_{ij} \,
		\mathcal{F}^{-1}\big( i k_j \,\mathcal{F}[\dot p] \big)
		\Big]\Bigg).
	\end{aligned}
	\]
	
	All coefficient fields (e.g., $\rho^{\mathrm{eff}}_{ij}$, $C_{JK}$, $\alpha_J$, and $B_{ij}$) are applied pointwise in physical space, while the Fourier transform is used exclusively to evaluate spatial derivatives.

	\subsection{Matrix-free implementation of the splitting scheme}
	
	The operator system is written in block form as
	\[
	A = \begin{pmatrix} A_{11} & A_{12} \\ A_{21} & A_{22} \end{pmatrix}.
	\]
	In the $L$-stabilized splitting scheme, each iteration requires solving
	\[
	\begin{aligned}
		(A_{11}+L)\mathbf{v}^{k+1} &= \mathbf{f}^\omega - A_{12}\dot{p}^k + L \mathbf{v}^k,\\
		A_{22}\dot{p}^{k+1} &= q^{\omega}-A_{21}\mathbf{v}^{k+1}.
	\end{aligned}
	\]
	
	These systems are solved in a matrix-free manner using Krylov methods (e.g.\ GMRES \cite{saad1986gmres}), where each operator application is computed via FFT-based differentiation combined with pointwise coefficient multiplication.
	
	\paragraph{Algorithm}
	Given $(\mathbf{v}^k,\dot{p}^k)$:
	\begin{itemize}
		\item Compute $\text{RHS}_v = \mathbf{f}^\omega - A_{12}\dot{p}^k + L\mathbf{v}^k$ and solve $(A_{11}+L)\mathbf{v}^{k+1} = \text{RHS}_v$.
		\item Compute $\text{RHS}_p = -A_{21}\mathbf{v}^{k+1}$ and solve $A_{22}\dot{p}^{k+1} = \text{RHS}_p$.
	\end{itemize}
	
	The dominant cost arises from repeated operator evaluations within the Krylov iterations, each involving a small number of FFT and inverse FFT operations.
	
	\subsection{Practical considerations and limitations}
	
	For variable coefficients, pseudo-spectral evaluation introduces aliasing errors due to nonlinear interactions in Fourier space. These are mitigated using standard techniques such as $2/3$-rule truncation or spectral filtering \cite{Fornberg1987,Kosloff1982}. Adequate resolution and filtering are essential for stability and accuracy.
	
	Absorbing boundary conditions are implemented using perfectly matched layers (PML) via complex coordinate stretching \cite{berenger1994,chew1997}, which effectively suppress spurious reflections.
	
	Compared to FEM, the pseudo-spectral method avoids assembly of large sparse matrices and instead relies on FFT-based operator evaluation, resulting in reduced memory usage and efficient parallelization on structured grids.
	
	While highly efficient for smooth or layered media, pseudo-spectral methods are less flexible for complex geometries, where FEM may be preferable. For constant coefficients, operator identities from the continuous formulation are preserved exactly; for variable coefficients, they hold up to discretization and aliasing error, provided sufficient resolution is used.

		\bibliography{refs}

\end{document}